\documentclass[12pt]{amsart}
\usepackage{amsfonts}
\usepackage{amsxtra}
\usepackage{amsthm}
\usepackage{amsfonts, amscd}
\usepackage{xypic, xy}
\usepackage{amssymb,latexsym,eufrak,amsmath,amscd, graphicx, color}
\usepackage{pgf,tikz}
\usetikzlibrary{arrows}

\def\ZZ{{\mathbb Z}}
\def\NN{{\mathbb N}}

\def\CC{{\mathbb C}}
\def\AA{{\mathbb A}}
\def\RR{{\mathbb R}}
\def\QQ{{\mathbb Q}}
\def\PP{{\mathbb P}}

\def\cI{\mathcal{I}}
\def\cE{\mathcal{E}}

\def\cA{\mathcal{A}}

\def\cG{\mathcal{G}}
\def\cF{\mathcal{F}}
\def\cC{\mathcal{C}}
\def\cL{\mathcal{L}}
\def\cO{\mathcal{O}}

\def\cN{\mathcal{N}}
\def\cU{\mathcal{U}}
\def\cV{\mathcal{V}}
\def\cT{\mathcal{T}}

\DeclareMathOperator{\Coker}{Coker}

\DeclareMathOperator{\Spec}{Spec}

\DeclareMathOperator{\Supp}{Supp}

\newtheorem{lemma}{Lemma}[section]
\newtheorem{theorem}[lemma]{Theorem}
\newtheorem{corollary}[lemma]{Corollary}
\newtheorem{proposition}[lemma]{Proposition}

\theoremstyle{definition}
\newtheorem{definition}[lemma]{Definition}

\newtheorem{remark}[lemma]{Remark}
\newtheorem{notation}[lemma]{Notation}

\numberwithin{equation}{section}

\newcommand{\bean}{\begin{eqnarray}}
\newcommand{\eean}{\end{eqnarray}}
\newcommand{\be}{\begin{displaymath}}
\newcommand{\ee}{\end{displaymath}}
\newcommand{\bea}{\begin{eqnarray*}}
\newcommand{\eea}{\end{eqnarray*}}

\newcommand{\ol}{\overline}

\begin{document}

\title{Gromov-Witten invariants for varieties with $\CC^*$ action}

\author{Anca~Musta\c{t}a, Andrei~Musta\c{t}\v{a}}

\address{School of Mathematical Sciences, University College Cork, Ireland}

\begin{abstract} For any smooth projective variety with a $\CC^*$ action, we reduce the problem of computing its Gromov-Witten invariants to the similar problem for its fixed locus. Starting from the stacky version of variation of GIT for our variety, we construct  the building blocks for the fixed loci of the moduli space of stable maps. We use this construction to compute their contribution to the virtual fundamental class.

\end{abstract}

\email{{\tt andrei.mustata@ucc.ie}}
\maketitle

\bigskip

\section{Introduction}

For any compact manifold with a one parameter group action, the Atiyah-Bott Localization Theorem allows one to recover an equivariant cohomology class from its pullbacks to the fixed loci.  In the algebraic context a similar result was established by D.Edidin and W.Graham in \cite{edidingraham}. Starting from the first computations of Gromov-Witten invariants, localization proved to be a powerful tool in Gromov-Witten theory. Not long after the introduction of a rigorous algebraic-geometric definition for virtual fundamental classes by K.Behrend and B.Fantechi (see \cite{kai}),
T.Graber and R.Panharipande (\cite{graber_pand}) proved the virtual localization formula
 for moduli spaces with a $(\CC^*)^k$ equivariant perfect obstruction theory.
The main example of such moduli spaces at the time was the moduli space of stable maps for smooth projective varieties with a torus action.

However until now, this method has been applied only for varieties with strong torus action, i.e. with only finitely many  orbits in dimensions 0 and 1. In this paper we develop a general procedure by which localization can be applied to compute equivariant Gromov-Witten invariants for a more general class of smooth projective varieties $X$ with  $\CC^*$ actions.

Virtual localization reduces the virtual fundamental class $[\overline{M}_{g,n}(X,\beta)]^{vir}$ to a sum in the expected equivariant Chow group
\bea [\overline{M}_{g,n}(X,\beta)]^{vir}= i_*\sum \frac{[F_\Gamma]^{vir}}{e^{\CC^*}(N_{F_{\Gamma }}^{vir})}\eea
where the summands correspond to the components $F_\Gamma$ of the fixed point locus of $\overline{M}_{g,n}(X,\beta)$. The classes ${[F_\Gamma]^{vir}}$ and ${e^{\CC^*}(N_{F_{\Gamma }}^{vir})}$  are obtained from the \emph{fixed} and \emph{moving} parts of the obstruction theory restricted to $F_\Gamma$.

Following the procedure developed in the case of target varieties with strong torus action, the components of the fixed point locus in the moduli space of stable maps can be indexed by decorated graphs. Each vertex $v$ in such a graph contributes a moduli space $\overline{M}_{g_v,n_v}(X_v,\beta_v)$, where $X_v$ is a component of the fixed locus of $X$. Each edge $(u, v)$ corresponds to invariant curves whose image stretches between the fixed loci $X_v$ and $X_u$. Thus an edge contributes a certain fixed component of a moduli space of stable maps in genus zero, with two marked points. We will denote such components by $M_{\omega, 2}$. The component of $\overline{M}_{g,n}(X,\beta)^{\CC^*}$ associated to a decorated graph $\Gamma$ is then a finite quotient of a fibre product of spaces mentioned above.

The virtual fundamental class of such a product splits into a product of contributions from the constituent moduli spaces, with additional special contributions of the cotangent classes $\psi$ at the nodes.
Thus computing the Gromov-Witten invariants of the $X$ via localization requires two types of inputs:
\begin{itemize}
\item[(1)] The Gromov-Witten invariants of the fixed point loci in X, with descendants;
\item[(2)] The contributions from the spaces $M_{\omega, 2}$.
\end{itemize}
Assuming the former known, we focus on the later. We will show how to calculate this contribution in terms of the following data:
\begin{itemize}
\item  The components of the fixed locus in $X$, their normal bundles and their decomposition induced by the $\CC^*$--action;
\item  The Bialynicki-Birula decompositions of $X$ given by the torus action;
\item A stacky version of the variation of GIT for $X$, as well as for its Bialynicki-Birula strata.
\end{itemize}
Of all the spaces $M_{\omega, 2}$, the most prominent is a component (or union of components) of the fixed point locus in the moduli space of stable maps with the homology class of the action map $t \to t\cdot x$ for generic $x\in X$. To explain how we constructed this particular $M_{\omega, 2}$, we start by considering the natural Deligne-Mumford stack structures of the GIT quotients of X, and their weighted blow-ups corresponding to variation of GIT. Taking fiber products we obtain an entire hierarchy of spaces, and $M_{\omega, 2}$ is the inverse limit of the induced  inverse system. Each of the spaces in the system represents a moduli problem,  and the universal families  of all these moduli spaces duly form a hierarchy of their own. Furthermore, each universal family comes with its own evaluation map  into a suitable target space birational to $X$. The universal family over $M_{\omega, 2}$ is the inverse limit of the inverse system generated by these target spaces and the birational morphisms between them.
The three networks thus obtained are connected into a network of triples, each triple consisting of a moduli space, universal family and evaluation map. This intricate structure allows us to compare the moving parts of the obstruction theories for these moduli spaces.

 Furthermore, each of the three networks is organized hierarchically, with the structure of the weighted blow-ups at the basis readily decipherable in terms of the Bialynicki-Birula strata in X, their normal bundles and the weights induced by the $\CC^*$--action.
This allows us to compute the moving part of the virtual fundamental class for the main space $M_{\omega, 2}$. The fixed part also results from the inverse system structure based on the formula proven in \cite{andrei} (Theorem 6.3).

The remaining spaces $M_{\omega, 2}$ can be divided into two categories. Some are obtained from the stacky points of the stacky GIT quotients mentioned earlier. More precisely, certain components of the inertia stacks of the above stacky quotients also form their own inverse systems, whose inverse limits are some of the desired spaces $M_{\omega, 2}$. A second category of spaces parametrizes maps whose class is not a multiple of the generic orbit. Such stable maps stretch between components of the fixed locus of $X$  other than the source and sink. In this case we employ stacky quotients of Bialinycki-Birula strata, which are pulled back through the networks of maps constructed above and  intersected to yield the second category of spaces $M_{\omega, 2}$.

The paper is organized as follows: In section 2 we set up some basic tools and notations required in indexing the components of the fixed locus of the moduli space of stable maps, and outline the main factors in the equivariant computation of the virtual fundamental class. In section 3 we describe the main  component $M_{\omega, 2}$. We set up the three networks of spaces, expressing each as canonical components of the fixed locus of certain moduli spaces of weighted stable maps. We also record the structure of the most important weighted blow-up morphisms in some detail. The fourth section is dedicated to computations of relevant classes in equivariant $K$-theory.

\section{The fixed locus of the moduli space of stable maps}

Let $X$ be a smooth projective variety with an algebraic $\CC^*$--action such that for a generic point $x \in X$, the stabilizer $ \mbox{ Stab}(x)$ is trivial.

The $\CC^*$- action on $X$ induces a natural $\CC^*$- action on the moduli space of stable maps $\overline{M}_{0,n}(X, \beta)$. Our first goal is to describe the fixed point locus $\overline{M}_{0,n}(X, \beta)^{\CC^*}$. We will present a way of indexing (unions of) components of $\overline{M}_{0,n}(X, \beta)^{\CC^*}$ by suitably chosen graphs.
We will start by defining the main building blocks.

We consider a $\CC^*$--equivariant embedding $ \iota: X \hookrightarrow \prod_{i=1}^m \PP^{n_i}$ and projections $\pi_i: \prod_{k=1}^m \PP^{n_k} \to \PP^{n_i}$
giving ample line bundles $\cL_i:= (\pi_i\circ \iota)^*\cO_{\PP^{n_i}}(1)$, such that the classes $c_1(\cL_i)$ with $i\in\{1,...,n\}$ form a basis for $H^{1,1}(X)$.

For the $\CC^*$--action \bea   \CC^*\times \PP^{n_i} & \to &  \PP^{n_i} \mbox{ given by }
\\ (t, [z_0:...:z_{n_i}]) &\to & [t^{a_0}z_0:...: t^{a_{n_i}}z_{n_i}] \eea
we define the moment map $\mu_i: \PP^{n_i} \to \RR$ by \bea  \mu_i([z_0:...:z_{n_i}]) := \frac{ \sum_{j=0}^{n_i} a_j |z_j|^2}{ \sum_{j=0}^{n_i} |z_j|^2}. \eea
 We note that $\mu_i$ maps the $\CC^*$--fixed locus of $\PP^{n_i}$ onto the set of weights $\{a_0, ..., a_{n_i}\}$.

\begin{remark} \label{degree of orbit maps} For each $p\in \PP^{n_i}$, the function $\nu:\CC^* \to \RR$ given by $\nu(t)= \mu_i( t\cdot p)$ is constant on circles $S^1(r)=\{t\in \CC^*\mbox{; } |t|=r\}$ while $\nu_{|\RR^*}$ is non-decreasing and
\bea  \lim_{t\to 0} \mu_i( t\cdot p) &=& \mbox{ min }\{a_j \mbox{; } j\in \{1,..., n_i\} \mbox{ and } z_i\not= 0\}=: m(p),  \\
 \lim_{t\to \infty} \mu_i( t\cdot p) &=& \mbox{ max }\{a_j \mbox{; } j\in \{1,..., n_i\} \mbox{ and } z_i\not= 0\}=: M(p),\eea
 so that the closure $\overline{O_p}$ of the orbit $O_p=\CC^* p$ satisfies $ \overline{\mu_i(O_p)}=[m(p), M(p)]$.

Furthermore, the map $\CC^* \to \PP^{n_i}$ given by $ t \to  t\cdot p $
      extends uniquely to a map $f_p: \PP^1 \to \PP^{n_i}$ satisfying
\bea  \deg f_{p *}[\PP^1]= \deg  f_{p *}(\cO_{\PP^{n_i}}(1)) = M(p)- m(p). \eea
\end{remark}

 Putting together the properties of $\mu_i$ above we obtain the following

\begin{lemma}\label{classes of orbit maps}
Consider the map  $\mu^m:  \prod_{i=1}^m \PP^{n_i} \to \RR^m$ given by $\mu^m = (\mu_1, ..., \mu_m)$, and the isomorphism  $\ZZ^m \to (H^{1,1}(X))^V \bigcap H_2(X, \ZZ)$ given by the dual of the basis $\{ c_1(\cL_i)\mbox{; } i\in\{1,...,n\}\} $ of $H^{1,1}(X)$ associated to the embedding $ \iota: X \hookrightarrow \prod_{i=1}^m \PP^{n_i}$.

For each $x\in X$, let $x_0:=\lim_{t\to 0}  t\cdot x$ and $x_\infty:=\lim_{t\to \infty} t\cdot x$ and let $\mbox{ Stab }(x)$ denote the stabilizer of $x$.
\begin{itemize}
\item[a)] The components of the fixed locus of $X$ under the $\CC^*$-- action are mapped by $\mu^m:  \prod_{i=1}^m \PP^{n_i} \to \RR^m$ into points on the lattice $\ZZ^m$, with coordinates given by the weights of the $\CC^*$--action on the $\PP^{n_i}$-s.
\item[b)] The map $ \CC^* \to X $ given by $  t \to  t\cdot x$
      extends uniquely to a $|\mbox{ Stab}(x)|: 1$ map $f_x: \PP^1 \to X$ of class $f_{x *}[\PP^1]$ represented as a vector
      \bea  f_{x *}[\PP^1] = \mu^m(x_\infty) -\mu^m(x_0) \in \ZZ^m, \eea while the closure of the orbit $O_x$ satisfies
      \bea  [\overline{O_x}] = \frac{\mu^m(x_\infty) -\mu^m(x_0)}{|\mbox{ Stab}(x)|}.\eea

\end{itemize}
\end{lemma}
\begin{proof} Via the projections $\pi_i: \prod_{k=1}^m \PP^{n_k} \to \PP^{n_i}$, the proof is reduced to the observations on $\mu_i$ in Remark \ref{degree of orbit maps}.  \end{proof}


We define a partial order on $\RR^m$ as follows: $u < v  \Longleftrightarrow u_i < v_i $  for all $i\in \{1,..., n\}$. Based on the following observations, we can construct an oriented graph with vertices in $\ZZ^m$ as follows:

\begin{definition} \label{graph of the action} \textbf{The oriented graph associated to a $\CC^*$-action on $X$:}

We define a set of vertices $\cV \subset \ZZ^m$ as the image through $\mu^m$ of the fixed point locus of $X$. For each $u \in \cV$, we denote by \bea X_u=\{x \in X \mbox{;  } \mu^m(x)=u \mbox{ and } t\cdot x=x \mbox{ for all }t \in \CC^*\}. \eea

 We define a set of edges $\cE \subset \cV \times \cV$ as follows: a pair $(u,v) \in \cE$ if and only if
 \begin{itemize}\item $u < v$, and
 \item there is a $\CC^*$--orbit $O_x\subseteq X$ such that $u= \mu^m(x_0)$ and  $v= \mu^m(x_\infty)$ . \end{itemize}
\end{definition}

\begin{definition}\label{invariant curve classes} \textbf{The invariant curve classes associated to the graph $(\cV, \cE)$:}
 Let $\Omega$ denote the set of pairs $(c,\beta)$ with $c \in \cV \cup \cE$ and $\beta \in H_2(X, \ZZ)$ defined by the following conditions:
  \begin{itemize}
  \item[(i)] If $c \in \cV$, then $\beta\in j_{c *} H_2(X_c, \ZZ)$ where $j_c: X_c \hookrightarrow X$.
  \item[(ii)] If $c=(u, v)\in \cE$, then $\beta=k[\overline{O_x}]$ where the orbit $O_x= \{t\cdot x|t\in \CC^*\}$ satisfies  $u, v \in \overline{\mu^m(O_x)}$, and $k\in \NN$. \end{itemize}
 \end{definition}

 By Lemma \ref{classes of orbit maps}, we have $\beta=k[\overline{O_x}]= \frac{k}{|\mbox{ Stab}(x)|} (v-u)\in \QQ^m.$

\begin{notation}
  For each $\omega \in \Omega$,  we will denote its second component by $\beta_\omega$. We will also denote by $X_\omega$ the fixed locus $X_u$ whenever $\omega= (u, \beta) \in \Omega.$
\end{notation}

\begin{notation}
      For each $\omega=(u, \beta) \in \Omega \bigcap (\cV \times H_2(X, \ZZ))$ and non-negative integer $n$, we denote  \bea   M_{\omega, n}(X):= \left\{ \begin{array}{l} \overline{M}_{0,n}(X_{u},\beta) \mbox{, if } n \geq 3, \mbox{ or } \beta\not=0, \\
     X_{u} \mbox{ otherwise. }    \end{array} \right. \eea
\end{notation}

\begin{definition} \label{pieces of the fixed locus}
      For each $\omega=((u, v), \beta) \in \Omega \bigcap (\cE \times H_2(X, \ZZ))$, we define $M_{\omega,2}(X)$ to be the substack of $\overline{M}_{0,2}(X, \beta)^{\CC^*}$ parametrizing isomorphism classes of stable maps with marked points $[(C, \varphi, (s_0, s_\infty))]$ fixed by $\CC^*$ which in addition satisfy the following properties:
      \begin{itemize}
     \item[(1)] $C=\bigcup_{i=1}^{s} C_{j}$ is a connected chain of $\PP^1$-s (possibly of length one). The marked points $s_0, s_\infty$ lie in the first and last components of the chain, respectively, and
      $\mu^m \circ \varphi (s_0)=u$, while $\mu^m \circ \varphi (s_\infty)=v$.
        \item[(2)] The map $\varphi$ is flow-preserving: Given the marked points $s_0, s_\infty=s_s$ and the nodes $\{ s_j \}= C_{j} \bigcap C_{{j+1}} $ for $j\in \{1,...,s-1\} $, their images $x_j:=\varphi (s_j)$ satisfy
        \bea   x_j= \lim_{t\to 0} t\cdot x \mbox{ and } x_{j+1}= \lim_{t\to \infty} t\cdot x \mbox{ for generic } x\in \varphi(C_{j}).  \eea
         \item[(3)] $\varphi_*[C]=\beta$ and the restriction of $\varphi$ on each component $C_{j}$ has class  \bea  \varphi_{*}[C_{j}]=\frac{\beta}{\mu^m (x_\infty )- \mu^m (x_0)}(\mu^m (x_{j})-\mu^m (x_{j-1})) \in \ZZ^m. \eea
      \end{itemize}
\end{definition}

 With the notations above, each of the irreducible components $\varphi(C_{j})$ is invariant under the $\CC^*$--action on $X$,
      the points $x_j=\varphi(s_j)$ are exactly all the fixed points of the $\CC^*$-action on $\mbox{ Im } \varphi$.

Note that $\frac{\beta}{\mu^m (x_\infty )- \mu^m (x_0)}= \frac{\beta}{(v- u)} \in \QQ$ due to Lemma \ref{classes of orbit maps} and Definition \ref{invariant curve classes}.

Condition (3) is chosen so as to insure that the spaces thus constructed are the smallest building blocks suitable for assembling subspaces of the fixed point loci, e.g. $M_{\omega,2}(\prod_i\PP^{n_i})$ is a connected component of $\overline{M}_{0,2}(\prod_i\PP^{n_i}, \beta)^{\CC^*}$.

\bigskip

 We will group components of the fixed locus $\overline{M}_{0,n}(X,\beta)^{\CC^*}$ in a way that is naturally compatible with the decomposition into components of  $\overline{M}_{0,n}(\prod_i\PP^{n_i},\beta)^{\CC^*}$. We index such (unions of) components by triples $(T, \tau, m)$ defined as follows:

\begin{definition} \label{triple0}
A \emph{minimal triple} $(T, \tau, m)$ is defined by the following properties:
\begin{itemize}
\item[(1)] $T$ is a tree, with $V(T)$ and $E(T)$ the sets of vertices and edges, respectively.
For each vertex $v\in V(T)$, we denote by $d(v)$ the degree of $v$, i.e. the number of edges incident to $v$.
\item[(2)]$m: V(T) \cup E(T) \to \ZZ_{\geq 0}$ is such that $n= \sum_{v \in V(T)} m(v) + \sum_{e \in E(T)}m(e)$.
\item[(3)]$\tau : V(T) \cup E(T) \to \Omega$ is a map whose composition with the projection $\Omega \to  \cV \cup \cE$ is a morphism of graphs, and satisfying
\bea   \beta = \sum_{v \in V(T)} \beta_{\tau(v)}+ \sum_{e \in E(T)} \beta_{\tau(e)}. \eea
\item[(4)] For each chain $\{v_1, v_2\}, \{v_2, v_3\},..., \{v_{s-1},v_s\}$ in $T$ such that $d(v_j)=2$ for $1<j<s$, if
\bea  \tau(v_1)=(u_1, \beta_1)\mbox{, }  \tau(v_k)=(u_k, \beta_k) \mbox{ and }  \tau(v_j)=(u_j, 0) \mbox{ for }  1<j<s, \eea
then for all possible elements $((u_1, u_s), \beta )\in \Omega$, there exists an index $j\in\{ 1,..., s-1\}$ such that
\bea  {\beta_{\tau(\{v_j, v_{j+1}\})}} \not= \frac{\beta}{u_{s}-u_1}{(u_{j+1}-u_j)}. \eea
\end{itemize}
\end{definition}

As before, we note that $((u_1, u_s), \beta )\in \Omega \Longrightarrow \frac{\beta}{u_{s}-u_1}\in \QQ$.

\begin{definition} \label{Triples}
 For each triple $(T,\tau,m)$  as above we associate a moduli space $F_{(T,\tau,m)}$
defined as the fibre product of stable map spaces given by the following fibre square diagram:
  \bea \diagram
F_{(T,\tau,m)} \dto^{} \rto^{ } & \prod_{v\in V(T)} M_{\tau(v),d(v)+m(v)} \times \prod_{e \in E(T)}M_{\tau(e),m(e)+2} \dto^{ \prod_{i\in S} \mbox{ev}_i} \\
\Delta:\prod_{v\in V(T)} \prod_{i=1}^{d(v)} X_{\tau(v)} \rto & \prod_{v\in V(T)} \prod_{i=1}^{d(v)} X_{\tau(v)}\times \prod_{e=\{v_1, v_2\} \in E(T)}( X_{\tau(v_1)}\times X_{\tau(v_2)}),
\enddiagram \eea
where  $S$ is constructed by selecting the first $d(v)$ marked points from  each of the factors of the product \bea  \prod_{v\in V(T)}\overline{M}_{0,d(v)+m(v)}(X_{{\tau(v)}},\beta_{\tau(v)})\eea and the last 2 marked points in each factor of $\prod_{e \in E(T)}M_{\tau(e),2}$, and $\mbox{ev}_i$ are the corresponding evaluation maps.

Note that $\sum_{v\in V(T)}d(v)=2|E(T)|$ and the factors of the product \bea \prod_{e\{v_1, v_2\}  \in E(T)}( X_{\tau(v_1)}\times X_{\tau(v_2)})\eea are, up to a permutation, the same as those in
$\prod_{v\in V(T)} \prod_{i=1}^{d(v)} X_{\tau(v)}$ so that the lower arrow ${\Delta}$ is the diagonal map followed by the permutation above.

\end{definition}


\begin{proposition} \label{fixpointsmoduli}
The fixed point locus $\overline{M}_{0,n}(X,\beta)^{\CC^*}$ is a disjoint union of  spaces $F_{(T,\tau,m)}/A_{(T,\tau,m)}$, where the moduli spaces $F_{(T,\tau,m)}$ are indexed by all triples $(T,\tau,m)$ as in Definition \ref{Triples}, and $A_{(T,\tau,m)}$ is a finite group (the automorphism group associated to the data $(T,\tau,m)$).
\end{proposition}

We will dedicate the next section to a constructive proof of the representability of the functor $M_{\omega, 2}$, and implicitly $F_{(T,\tau,m)}$. We will focus on the main case when $omega$ corresponds to the class of the action map $t\to t\cdot x$ for $x\in X$ generic. The remaining cases will derive from here.

\subsection{Maps from curves in higher genus.} For moduli spaces of stable maps in general genus $g$, Definitions \ref{triple0} and \ref{Triples} and Proposition \ref{fixpointsmoduli} can be modified accordingly to work with subspaces $F_{(\Gamma,\gamma,m)}$ of $\overline{M}_{g,n}(X,\beta)^{\CC^*}$. In this context $\Gamma$ is a graph and the function $\gamma: E(\Gamma) \bigcup V(\Gamma) \to \Omega\times \ZZ_{\geq 0}$ has an additional coordinate $g_\gamma$ indexing the genus of a curve component, such that $g$ is obtained by summing the values of $g_\gamma$ for all the vertices and edges of $\Gamma$, together with the number of loops of the graph $\Gamma$.  While every $v\in V(\Gamma)$ has an associated moduli space of stable maps $\ol{M}_{g_{\gamma (v)}, m(v)+d(v)}(X_{\gamma(v)}, \beta_{\gamma(v)})$ with $g_{\gamma(v)}\geq 0$, in the case of $e\in E(\Gamma)$, the associated space $M_{\gamma(e), m(e)+2}$ parametrizes maps with domains of genus 0. The focus of this article is on the contribution of these last spaces to the Gromov-Witten invariant calculations.

\subsection{Gromov-Witten invariants}


The virtual localization formula proved by Graber and Pandharipande in \cite{graber_pand} presents the virtual fundamental class $[ \overline{M}_{g,n}(X,\beta)]^{vir}$ as a sum in the Chow ring  $A^{\CC^*}_*(\overline{M}_{g,n}(X,\beta))\otimes \QQ[t, \frac{1}{t}]$:
\bea [\overline{M}_{g,n}(X,\beta)]^{vir}= i_*\sum \frac{[F_{(\Gamma, \gamma,m)}]^{vir}}{e^{\CC^*}(N_{F_{(\Gamma, \gamma,m)}}^{vir})}\eea
where
$e^{\CC^*}$ denotes the top equivariant Chern class.

For each $F_{(\Gamma, \gamma,m)}$,  the class $[F_{(\Gamma, \gamma,m)}]^{vir}$  and $N_{F_{(\Gamma, \gamma,m)}}^{vir}$ are defined as follows: Given a perfect obstruction theory $(E^{\bullet}, \phi)$ for $\overline{M}_{g,n}(X,\beta)$ with an equivariant lift of the $\CC^*$-action,
\begin{itemize}
\item The \emph{virtual class} $[F_{(\Gamma, \gamma,m)}]^{vir}$ is calculated from the fixed part of the  restriction $(E^{\bullet}_{(\Gamma, \gamma,m)}, \phi_{(\Gamma, \gamma,m)})$ of the obstruction theory $(E^{\bullet}, \phi )$ to $F_{(\Gamma, \gamma,m)}$. This is itself a perfect obstruction theory (\cite{graber_pand}).
\item The \emph{virtual normal bundle} $N_{F_{(\Gamma, \gamma,m)}}^{vir}$ is a two--term complex defined as the moving part of $E_{\bullet, (\Gamma, \gamma,m)}.$
\end{itemize}

Furthermore one can write an explicit formula for each the terms $\frac{[F_{(\Gamma, \gamma,m)}]^{vir}}{e^{\CC^*}(N_{F_{(\Gamma, \gamma,m)}}^{vir})}$ by applying the formula for the virtual class of the boundary strata from \cite{ber}. Let $V_2(T) \subset V(T)$ denote the degree 2 vertices $v$ of tree $T$, with adjacent edges $e^1_v$ and $e^2_v$, and such that $\beta_{\gamma (v)}=0$. Then
$\frac{[F_{(\Gamma, \gamma,m)}]^{vir}}{e^{\CC^*}(N_{F_{(\Gamma, \gamma,m)}}^{vir})}$ can be written as
$$\frac{ [\Delta] \prod_{v\in V(T)}[M_{\gamma (v),d(v)+m(v)}]^{vir} \prod_{e \in E(T)} [M_{\gamma (e),2}]^{vir}}{\prod_{e \in E(T)}e^{\CC^*}(N_{M_{\gamma (e),2}}^{vir}) \prod_{v \in V_2(T)}(\psi_{e^1_v,v}+\psi_{e^2_v,v})\prod_{v \in V(T) \setminus V_2(T),v \in e \in E(T)}(\psi_{e,v} +\psi_{v,i})}.$$
Here $i$ is a marked point for $M_{\gamma (v),d(v)+m(v)}$, chosen so that $v$ and $i$ have the same images through their respective evaluation maps. Then $\psi_{e,v}$ and $\psi_{v,i}$ denote the top equivariant Chern classes of the line bundles on $M_{\gamma (e),2}$ and $ M_{\gamma (v),d(v)+m(v)}$, respectively, given by cotangent directions along the universal curve, restricted to the marked point corresponding to $v$, and the $i$-th marked point, respectively. As before, $ [M_{\gamma (e),2}]^{vir}$ and $[M_{\gamma (v),d(v)+m(v)}]^{vir}$ denote the fixed part of the virtual class, respectively the virtual class of the corresponding moduli spaces.

For each $\omega=((u,v), \beta)\in \Omega$  we define $c_{\omega,n_1,n_2}: H^*(X_{u})[t,t^{-1}] \to H^*(X_{v})[t,t^{-1}]$ by $$c_{\omega,n_1,n_2}(\alpha)= ev_{2*}(ev_1^*(\alpha) \psi_1^{n_1}\psi_2^{n_2}\cap \frac{[M_{\omega,2}]^{vir}}{e^{\CC^*}(N_{M_{\omega,2}}^{vir})}).$$
Here $n_1, n_2 \in \ZZ_{\geq 0}$.
Computation of equivariant Gromov-Witten invariants with descendants for $X$ can now be reduced to the computations of Gromov-Witten invariants with descendents for the fixed locus of $X$, together with calculating $c_{\omega,n_1,n_2}$. In this paper we will show how to compute $c_{\omega,n_1,n_2}$ by describing both $[M_{\omega,2}]^{vir}$ and $N_{M_{\omega,2}}^{vir}$.

\section{A construction of the main subspace $M_{\omega, 2}$ of $\overline{M}_{0,2}(X,\beta)^{\CC^*}$.}

Consider a $\CC^*$-- equivariant embedding of $X$ into $\PP^N$ constructed by choosing a linearization on a very ample line bundle on $X$.
By a slight abuse of notation, we will denote by $\mu$ both the moment map $\mu: \PP^N \to \RR$ and its restriction to $X$. We can choose the linearization so that $\mu(\PP^N)=[0,d]$. (Changing the linearization modifies the moment map by a translation on $\RR$.) We may assume without loss of generality that $\mu(X)=[0,d]$. Indeed, if $\mu(X)=[i,i']$ with $0<i$ or $i'<d$, then the decomposition into strata induced by the $\CC^*$--action on $\PP^n$ would insure that $X$ is embedded in a projective subspace invariant under the $\CC^*$--action, and hence we may replace $\PP^n$ with this subspace in our discussion.

\begin{notation} \label{I} Let $I$ denote the image through $\mu$ of the fixed point locus of $\PP^N$. For each $i\in I$, let $P_i=\{x \in \PP^N \mbox{;  } \mu(x)=i \mbox{ and } t \cdot x=x \mbox{ for all }t \in \CC^*\}$ and  $X_i=P_i \cap X$.
\end{notation}


\subsection{Fixed loci in  moduli spaces of weighted stable maps}  Let us consider the case when $\beta$ is the class of the action map $t \to t\cdot x$ for generic $x\in X$. The oriented graph $(\cV, \cE)$ associated to the $\CC^*$--action (Definition \ref{graph of the action}) has a minimal and maximal vertices $u_0$ and $u_\infty$, connected to all other vertices by oriented paths. With the notations from the previous section, the images of the moment maps for $\prod_i\PP^{n_i}$ and $\PP^N$  are connected by a projection $\RR^m \to \RR$, which sends $u_0$ to 0 and $u_\infty$ to $d$.

Our goal is to provide a concrete construction for the moduli space $M_{\omega, 2}(X)$ introduced in Definition \ref{pieces of the fixed locus},  in the main case when $\omega =((u_0, u_\infty), \beta)$. The ambient space for this construction will be a moduli space of weighted stable maps.
The term of \emph{weighted stable maps} was coined in \cite{alexeev_guy} and \cite{bayer_manin}, where it referred to stable maps with weights on the marked points. We will be using this term in the more general sense where it refers to stable rational maps with weights on the marked points and the map (see for example \cite{noi3}).


We will start by defining special fixed loci in moduli spaces of weighted stable maps.
We will show that the simplest of these correspond to the GIT quotients of $X$ by $\CC^*$, and we will use the birational maps between them to construct $M_{\omega, 2}(X)$ as an inverse limit of this family of spaces.

We recall the definition of the moduli spaces of weighted stable maps as in \cite{noi3}:

\begin{definition} \label{weighted map spaces}
For a rational number $a>0$ and an $n$-tuple $\cA=(a_1,...,a_n)\in \QQ^n$  such that $0\leq a_j\leq 1$  and  $\sum_{j=1}^{n}a_j+da>2$, we define the moduli space of weighted stable maps
$\overline{M}_{0, \cA}(\PP^N, d, a)$ to parametrize isomorphism classes of tuples
\[    [( \pi \colon C \to B , \{s_j\}_{1\leq j\leq n}, \cL , e )]    \]
where $\pi: C\to B$ is a family of smooth (or at most nodal) rational curves, $s_j$ are sections of $\pi$ mapping to smooth points of the fibres,  $\cL$ is a line bundle on $C$ of degree $d$ on each fiber $C_b$, and $e:\cO_C^{N+1}\to \cL$ is a morphism of fiber bundles (specified up to isomorphisms of the target) such that:

\begin{enumerate}
\item $\omega_{C|B}(\sum_{j=1}^na_js_j)\otimes \cL^a$ is relatively ample over $B$,
\item $\cG :=\Coker e$, restricted over each fiber $C_b$, is a skyscraper sheaf supported only on smooth points, and
\item for any $s \in C_b$ and for any $ J \subseteq \{1,...,n\}$ (possibly empty) such that $ s = s_j(b) $ for all $j \in J$, the following condition holds  $$\sum_{\j \in J}  a_j + a\dim\cG_{s}\leq 1.$$
\end{enumerate}
\end{definition}

 For a more intuitive description of a weighted stable map, we note that the rational map $\varphi: C \dashrightarrow \PP^N$ determined by $e$, when restricted to each fibre $C_b$, extends naturally to a well-defined map $\varphi_b: C_b \to \PP^N$ of degree $\deg \varphi_b \leq d$. Whenever $\deg \varphi_b < d$, the curve $C_b$ contains some special base-locus points (the support of $\cG_b =\Coker e_b$), such that
\bean  \label{degree of weighted stable map} \deg \varphi_b + \sum_{p\in \mbox{Supp }(\cG_b)} \dim\cG_{p} = d. \eean
Each point $p\in \mbox{Supp }(\cG_b)$ comes with a positive multiplicity $m_p:=\dim\cG_{p}$. Condition (1) determines the minimum possible degrees of $\varphi_b$ on ending components ("tails") of $C_b$ in terms of $a$ and the weights of the marked points on the given components. Condition (2) determines the exact cases when the special sections $s_j$ can intersect and the maximum possible multiplicity $m_p$ of the intersection point. Conditions (1) and (2) are complementary: for a given subset of marked sections, the maximum possible multiplicity at the intersection is 1 less than the minimum possible degree on an ending component.

  There is a natural map $\overline{M}_{0, n}(\PP^N, d) \to \overline{M}_{0, \cA}(\PP^N, d, a)$. We denote by  $\overline{M}_{0, \cA}(X, \beta, a)$ the image of  $\overline{M}_{0, n}(X, \beta)$ under this map.

\begin{notation} Consider a real number $a$ satisfying $0<a \leq \frac{1}{d}$.
For each pair of elements $i, i' \in I$ with $i< i'$, we will be using the set of weights on two marked points
 $\cA(i,i')=( 1-i a, 1-(d-i')a)$.
\end{notation}

We note that due to the chosen weights, a point in  $\overline{M}_{0, \cA(i,i')}(\PP^N, d, a)$ parametrizes a chain $C$ of $\PP^1$-s whose starting and ending components $C_0$ and $C_\infty$ contain the marked points $s_0$ and $s_\infty$ respectively, together with a map $\varphi: C\to \PP^n$ non-constant on each component, and multiplicities $  m_0=\dim\cG_{0}$ and $m_\infty=\dim\cG_{\infty}$  such that  \bean \label{conditions on tails} \deg \varphi_{|C_0} + m_0 > i \geq m_0 \mbox{, } \deg \varphi_{|C_\infty} + m_\infty > d-i' \geq m_\infty.\eean



The $\CC^*$- action on $X$ and the equivariant embedding into $\PP^N$ also induce a natural $\CC^*$- action on  $\overline{M}_{0, \cA}(X, \beta, a)$. Indeed, a linearization on the very ample line bundle on $X$  gives an action on the bundle $\cO_C^{N+1}$, and hence on a weighted stable map. Note that any two linearizations determine the same action on a weighted stable map. Indeed, two different linearizations associated to the same action on $\PP^N$ differ by multiplication by a scalar on $\cO_C^{N+1}$ and hence on $\cL$.

\begin{definition}
\label{M} We denote the elements of $I$ by $0=i_0<i_1<...<i_k=d$.
  For each pair of indices $j,{j'}\in \{0, 1,..., k\}$ we denote by $M_{(j,{j'})}(\PP^N)$ the substack of $\overline{M}_{0, \cA(i_j,i_{j'})}(\PP^N, d, a)^{\CC^*}$  representing classes of weighted stable maps \bea [C\to B, (s_0, s_\infty ), \cL, e ] \eea fixed by $\CC^*$ such that each fibre $C_b=\bigcup_{l=0}^{r} C_{l}$ is a connected chain of $\PP^1$-s  in addition satisfying the following conditions:
  \begin{itemize}
   \item[(1)] The weighted stable map $\varphi_b$ determined on each fibre $C_b$ by $e$ is \emph{flow-preserving}: the images $x_l:=\varphi_b (s_l)$ of the marked points $s_0, s_\infty=s_r$ and the nodes $\{ s_l \}= C_{l} \bigcap C_{{l-1}} $   satisfy  $ x_l= \lim_{t\to 0} t\cdot x$  and  $ x_{l+1}= \lim_{t\to \infty} t\cdot x $ for generic $ x\in \varphi_b(C_{l}).$
   \item[(2)] The map $\varphi_b$ is \emph{action-class-preserving}: $\deg \varphi_* [C_l]=\mu (x_{l+1})-\mu (x_l)$ for each component $C_l$.
     \end{itemize}

We define  $M_{(j,{j'})}(X): = M_{(j,{j'})}(\PP^N) \times_{\overline{M}_{0, \cA(i_j,i_{j'})}(\PP^N, d, a)}{\overline{M}_{0, \cA(i_j,i_{j'})}(X, \beta, a)} $.

 Let $\cU_{(j,{j'})}(X)$ denote the universal family over $M_{(j,{j'})}(X)$.

\end{definition}

In the case of spaces $M_{(j,j+1)}$,  condition (2) is superfluous as in this case a weighted stable map $\varphi: C \to X$ maps $C\cong \PP^1$ to just one orbit in $X$.

\begin{lemma} \label{bottom_line}
 Let $i_{j} < i_{j+1}$ be two consecutive elements of $I$.

 For each closed point $b$ in $M_{(j,j+1)}(X)$, the fibre over $b$ of $\pi: \cU_{(j,j+1)}(X) \to M_{(j,j+1)}(X)$ is isomorphic to $\PP^1$ and the corresponding weighted stable map $\varphi_b$ satisfies
 \bea        \mbox{ Im } \mu\circ\varphi_b \supset [i_j, i_{j+1}].   \eea
\end{lemma}

\begin{proof}
Let $(C, (s_0, s_\infty ), \cL, e )$ be a weighted stable map such that $C$ is the fibre of the universal family over a closed point $b$ in $\overline{M}_{0, \cA(i_j,i_{j+1})}(X, \beta, a)$, and let $\varphi: C \to X$ be the well-defined map on $C$ induced by $e$.
Assume that $[(C, (s_0, s_\infty ), \cL, e )]$ is fixed by the $\CC^*$-action induced from the action on $X$. Thus for each $t\in \CC^*$ there exists an automorphism $g_t$ of $C$ and an isomorphism $\widetilde{g_t}: \cL \to  g_t^*\cL$ making the following diagram commutative:
 \bean
 \diagram
\cO_{C}^{N+1} \dto^{t\cdot} \rto^{ e } &  \dto^{\widetilde{g_t}} \cL\cong \cO_C(d) \\
  g_t^*\cO_{C}^{N+1}    \rto^{g_t^*e} &  g_t^*\cL\cong \cO_C(d).
\enddiagram \eean
 It follows that  $\mbox{ Im } \varphi$ is $\CC^*$-equivariant, and hence $\varphi (\Supp (\Coker e))$ is a finite subset of the fixed point locus of $X$. Thus $\Supp (\Coker e) \subseteq \{ s_0, s_\infty\}$. Indeed:
 \begin{itemize}
 \item No component of $C$ is mapped by $\varphi$ to the fixed point locus of $X$, due to condition (2) in Definition \ref{M}. (Note that condition (1) in Definition \ref{weighted map spaces}, together with the choice of weights insure that no  component of $C$ is contracted to a point by $\varphi$.) Thus each component can have at most two points mapped to the fixed point locus of $X$. Nodes are such points.
 \item No nodes can be in  $\Supp (\Coker e)$ by the definition of weighted stable maps. \item $\varphi(s_0)$ and $\varphi(s_\infty)$ are fixed points.
 \end{itemize}

Consider now a component $C_j\cong \PP^1$ of $C$. Locally around the point $0\in \PP^1$, the diagram above is of the form

$
\diagram
\CC[z]_{(z)}^{N+1} \dto \rto &  \dto \CC[z]_{(z)} \\
  \CC[z]_{(z)}^{N+1}    \rto &  \CC[z]_{(z)}
\enddiagram $ defined on basis elements by $\diagram
e_i \dto \rto &  \dto z^{\alpha_i}c_i \\
  t^{a_i}e_i    \rto &  t^{a_i}z^{\alpha_i}c_i = (t^rz)^{\alpha_i}c_i
\enddiagram$\\
where $c_i\in \CC$ must be a constant to insure diagram commutativity and $r\in \ZZ$ is determined by the isomorphism $g_t$.
It follows that $a_i= r\alpha_i$ for all indices $i$ for which $c_i\not=0$.

With these data we have $\deg\varphi_*[C_l]= M-m$ where $M=\mbox{ max }\{ \alpha_i\mbox{; } c_i\not=0\}$ and $m=\mbox{ min }\{ \alpha_i\mbox{; } c_i\not=0\}$. On the other hand the points $x_l:=\varphi(0)$ and $x_{l+1}=\varphi(\infty)$ satisfy $\mu(x_l)= \mbox{ min }\{ a_i\mbox{; } c_i\not=0\}$, $\mu(x_{l+1})= \mbox{ max }\{ a_i\mbox{; } c_i\not=0\}$. Thus condition (2) in Definition \ref{M} implies $r=1$, hence also $\mu(x_l)=\dim \cG_{x_l}$ and $d-\mu(x_{l+1})=\dim \cG_{x_{l+1}}$.

We are now ready to prove that $C\cong \PP^1$ and \bea \mu\circ\varphi (C) = [\mu\circ\varphi(s_0), \mu\circ\varphi(s_\infty)] \supseteq [i_j, i_{j+1}]. \eea

Indeed, with the notations from Definition \ref{M},  in the cases $l=0$ and $l=r$  inequalities (\ref{conditions on tails}) imply $\mu(x_0) \leq i_j$ and $\mu(x_\infty)\geq i_{j+1}$ as well as
\bea \mu(x_1) &=& \mu(x_0)+ \deg \varphi_{|C_0} >i_j \mbox{ and } \\
\mu(x_{r}) &=& \mu(x_\infty)+ \deg \varphi_{|C_r} < i_{j+1}.
\eea
 As $i_j$ and $i_{j+1}$ are consecutive elements in $I$, it follows that $\mu(x_1) \geq i_{j+1}$ and  $\mu(x_{r}) \leq i_j$.
 But the map $\varphi$ is flow-preserving and $i_j<i_{j+1}$,  so it must be that $r=1$ so that $x_1=\varphi(s_\infty)$, $x_{r}=\varphi(s_0)$ and $C\cong \PP^1$.


\end{proof}

Recall the following

\begin{theorem} (Theorem 8.3 in \cite{mumford} for the GIT quotient associated to a point in the chamber $(i_j,i_{j+1})$.) Let $i_j< i_{j+1}$ be two consecutive points in  $I$. Then there exists a linearization on the chosen very ample line bundle on $X$ such that the sets of semistable and stable points are given by \bea X^s_{(j,j+1)}=X^{ss}_{(j,j+1)}= \{ x\in X\mbox{; } (i_j, i_{j+1}) \subset \mu(O_x)\}. \eea
\end{theorem}

 Thus the previous Lemma we have proven that any weighted stable map $\varphi$ parametrized by a point in $M_{(j, j+1)}(X)$ satisfies $\varphi(C) \subset X^s_{(j,j+1)}$.

\begin{theorem} \label{consecutive}
If $i_j < i_{j+1}$ are consecutive elements of $I$, then $M_{(j,j+1)}(X)$ is isomorphic to the stacky quotient $[X^s_{(j,j+1)}/\CC^*]$ associated to the GIT quotient $X^s_{(j,j+1)}//\CC^*$.
\end{theorem}

\begin{proof} Step 1. We build a map $[X^s_{(j,j+1)}/\CC^*]\to M_{(j,j+1)}(X)$. To construct a family of weighted stable maps over $[X^s_{(j,j+1)}/\CC^*]$, we first compactify  $X^s_{(j,j+1)}$ by adding zero and infinity sections as follows: We consider the actions of $\CC^*$ on $\AA^1$ of weights 1 and -1, respectively:
\bea \CC^*\times \AA^1 \to \AA^1 &\mbox{ and }& \CC^*\times \AA^1 \to \AA^1, \\
(t, u) \to t u &\mbox{ and }& (t, u) \to t^{-1} u. \eea
We define $\overline{X^s}_{(j,j+1)}$ to be the stack obtained by gluing the two stacky quotients  $[(X^s_{(j,j+1)}\times \AA^1)/\CC^*]$ along $X^s_{(j,j+1)}= [(X^s_{(j,j+1)}\times \CC^*)/\CC^*]$. The natural map $\pi: \overline{X^s}_{(j,j+1)} \to [X^s_{(j,j+1)}/\CC^*]$ admits two sections $s_0$ and $s_\infty$ induced by
the zero-section $[X^s_{(j,j+1)}/\CC^*] \to [(X^s_{(j,j+1)}\times \AA^1)/\CC^*]$ in the two patches above.

           The inclusion $X^s_{(j,j+1)}\hookrightarrow X$ gives a rational map $\varphi: \overline{X^s}_{(j,j+1)} \dashrightarrow X$ associated to a morphism of vector bundles $e: \cO^{N+1} \to \cL$ over $\overline{X^s}_{(j,j+1)}$. The line bundle $\cL$ is the pull-back of $\cO_{\PP^N}(1)$ and is of degree $d$, the degree of the action map $t\to t\cdot x$ for $x$ generic.
           Condition (1) in Definition \ref{weighted map spaces} is satisfied due to the inequality $i_j<i_{j+1}$.  The support of $\Coker e$ is included in  $\mbox{ Im } s_0 \bigcup \mbox{ Im } s_\infty$  (Condition (2)).

            When restricted over each fibre $C_b=\PP^1$ of $\pi$, the rational map $\varphi$ is of the form $t \to t \cdot x$ and extends to a map $\varphi_b=f_x: \PP^1 \to X$, so that Lemma \ref{degree of orbit maps} and equation \ref{degree of weighted stable map} together imply
           \bean  \label{degree of fibre in $X_{i_1,i_2}$}  \dim \Coker e_{s_0(b)} +   \dim \Coker e_{s_\infty(b)} + \mu (x_\infty) - \mu (x_0) = d, \eean
where $x_0=\varphi_b(s_0(b))$, $x_\infty=\varphi_b(s_\infty(b))$ and therefore $\mbox{ Im } \mu\circ \varphi_b = [ \mu(x_0), \mu(x_\infty)]$.

We claim that
\bean \label{weights on sections} \dim \Coker e_{s_0(b)}= \mu(x_0) \mbox{ and } \dim \Coker e_{s_\infty(b)} = d- \mu(x_\infty). \eean

Consider fibres $C_b=\PP^1$  of $\pi$ such that $\Coker e_{s_0(b)} \not=0$. The generic fibres among these satisfy $\Coker e_{s_\infty(b)}=0$ and $\mu(x_\infty)=d$, and hence $\dim \Coker e_{s_0(b)}= \mu(x_0)$ by equation \ref{degree of fibre in $X_{i_1,i_2}$}. Similarly, the generic fibres among those where $\Coker e_{s_\infty(b)} \not=0$ satisfy $\Coker e_{s_0(b)}=0$ and $\dim \Coker e_{s_\infty(b)} = d- \mu(x_\infty)$. Since $\dim \Coker e_{s_\infty}$ and $\dim \Coker e_{s_\infty}$ are upper semi-continuous, it follows that for any fibre $C_b$ we have \bea  \dim \Coker e_{s_0(b)} \geq  \mu(x_0) \mbox{ and } \dim \Coker e_{s_\infty(b)} \geq d- \mu(x_\infty). \eea
 In conjunction with equation \ref{degree of fibre in $X_{i_1,i_2}$}, this yields equations \ref{weights on sections}.

Finally, we recall the condition $(i_j, i_{j+1}) \subset \mu(O_x)$ for $x\in X^s_{(j,j+1)}$. On the other hand, for $x\in \mbox{ Im } \varphi$, $\mu(O_x) = ( \mu(x_0), \mu(x_\infty))$ and hence \bea  \mu(x_0) \leq i_j < i_{j+1} \leq \mu(x_\infty). \eea
Together with equation \ref{weights on sections}, this suffices to check Condition (3) from Definition \ref{weighted map spaces}.

We have proven the existence of a map $[X^s_{(j,j+1)}/\CC^*]\to M_{(j,j+1)}(X)$, which is part of a fibre square

 \bea \diagram
\overline{X^s}_{(j,j+1)} \dto^{} \rto^{ \phi } &  \dto^{} \cU_{(j,j+1)}(X) \\
 [X^s_{(j,j+1)}/\CC^*] \rto^{} & M_{(j,j+1)}(X)
\enddiagram \eea

Step 2.  The rational map $\cU_{(j,j+1)}(X) \dashrightarrow X$ corresponds to a well defined map  $\cU_{(j,j+1)}(X)\setminus (\mbox{ Im } s_0\bigcup \mbox{ Im } s_\infty ) \to X^s_{(j,j+1)} $ which extends to $\cU_{(j,j+1)}(X) \to \overline{X^s}_{(j,j+1)}$, the inverse of the map $\phi$ above.

Indeed, consider a family of weighted stable maps $[C \to B, (s_0, s_\infty ), \cL, e ]$ given by a map  $B \to M_{(j,j+1)}(X)$, and let $\varphi: C \dashrightarrow X$ be the corresponding rational map. Let $q\in C_b$ be a point where $\varphi$ is not defined. The restriction $\varphi_{|C_b}$ extends to a map $\varphi_b: C_b\to X$. Since $[C \to B, (s_0, s_\infty ), \cL, e ]$ is fixed by $\CC^*$, it follows that the $\CC^*$--orbit $O_{\varphi_b(q)} \subset \mbox{ Im } \varphi_b$, and  $\varphi_b^{-1}(O_{\varphi_b(q)}) \subseteq \mbox{ Supp} \Coker e_b$. But $\Coker e_b$ is a skyscraper sheaf, so $\varphi_b(q)$ must be a fixed point. Hence $\cU_{(j,j+1)}(X)\setminus (\mbox{ Im } s_0\bigcup \mbox{ Im } s_\infty ) \to X$  is well defined, and the image is in $X^s_{(j,j+1)} $ by Lemma \ref{bottom_line}.

The map $\cU_{(j,j+1)}(X)\setminus (\mbox{ Im } s_0\bigcup \mbox{ Im } s_\infty ) \to X^s_{(j,j+1)} $ is an isomorphism of $\CC^*$-- bundles which extends canonically to an isomorphism of $\PP^1$--bundles. By naturality, this is the inverse of the map $\phi$ constructed above.
\end{proof}

\begin{remark}
 By applying the arguments of Lemma \ref{bottom_line} more generally for $j<j'$, we can characterize a closed point $[\varphi: C \to X, s_0, s_\infty]$ in $M_{(j,j')}(X)$ as follows: $C$ is a chain of $\PP^1$-s, on each of which $\varphi $ is the action map $t\to t\cdot x$ for some $x\in X$, and such that, via the moment map $\mu$,
 \begin{itemize}
 \item $\mu (\varphi(s_0)) \leq i_j$ and $\mu (\varphi(s_\infty)) \geq i_{j'}$;
 \item all the nodes of $C$ are mapped between $i_j$ and $i_{j'}$.
 \end{itemize}
\end{remark}

\begin{corollary}
 $M_{(j,j+1)}(\PP^N)$ is a closed and open substack of $\overline{M}_{0, \cA(i_j,i_{j+1})}(\PP^N, d, a)^{\CC^*}$.
 \end{corollary}

\begin{proof} The proof above implies that $M_{(j,j+1)}(X)\cong [{X^s}_{(j,j+1)}/\CC^*]$ is indeed a closed substack of $\overline{M}_{0, \cA(i_j,i_{j+1})}(X, \beta, a)$. Furthermore, we know  $\overline{M}_{0, \cA(i_j,i_{j+1})}(\PP^N, d, a)$  is a smooth stack and hence $\overline{M}_{0, \cA(i_j,i_{j+1})}(\PP^N, d, a)^{\CC^*}$ is smooth, its connected components are irreducible. It remains to note that for a generic point $p$ in $M_{(j,j+1)}(\PP^N)$, all points in an open neighbourhood of $p$ in $\overline{M}_{0, \cA(i_j,i_{j+1})}(\PP^N, d, a)^{\CC^*}$ represent weighted stable maps which also satisfy conditions (1)-(2) in Definition \ref{M}.

\end{proof}



\begin{notation} \label{action_strata} For each element $i_j\in I$, we denote by $X_j:=\mu^{-1}(i_j)^{\CC^*}$ the corresponding component of the fixed point locus of $X$, and
\bea  X_j^{-} := \{ x\in X \mbox{; } \lim_{t\to \infty } t\cdot x \in X_j\} \mbox{ and } X_j^{+} := \{ x\in X \mbox{; } \lim_{t\to 0 } t\cdot x \in X_j\}.   \eea
The natural maps $\pi_{j}^- :X_j^{-} \to X_j$ and $\pi_{j}^+ :X_j^{+} \to X_j$, obtained by taking limits $x\to \lim_{t\to 0} t\cdot x$ and  $x\to \lim_{t\to \infty} t\cdot x$, are affine bundles over $X_j$.
\end{notation}

It is known that the normal bundles satisfy $\cN_{X^-_j|X}=\pi_{j}^{- *}\cN_{X_j|X_j^+}$ and $\cN_{X^+_j|X}=\pi_{j}^{+ *}\cN_{X_j|X_j^-}$. The $\CC^*$ action on $X$ induces two actions on $\cN_{X^-_j|X}$: one which makes the projection $\cN_{X^-_j|X} \to X^-_j$ equivariant, and a second one induced from the action on $\cN_{X_j|X_j^+}$. The first action allows $\cN_{X^-_j|X}$ to descend to a bundle $[\cN_{X^-_j)|X}/\CC^*]$ over $[(X^-_j\setminus X_j)/\CC^*]$. In the second case, $\CC^*$ acts on each fibre, leading to a decomposition of the bundle by eigenvalues (weights). This decomposition descends to $[\cN_{X^-_j|X}/\CC^*]$.
Similarly for $\cN_{X^+_j|X}$.

Next we describe morphisms between spaces $M_{(j,j')}(X)$ as weighted blow-ups.

The local structure of a weighted blow-up $\widetilde{Z} \to Z$ of smooth Deligne-Mumford stacks has been described in detail in \cite{noi5}, section 2, and we will employ the notations introduced there. In particular we recall the associated affine
fibration   \bea A:=\Spec \left(\oplus_{n\geq 0} \cI_n/\cI_{n+1}\right) \to Y \eea
for a weighted blow-up $\widetilde{Z} \to Z$ along the locally embedded smooth stack $Y$, and the filtration $\{\cI_n\}_{n\geq 1}$ of the ideal of $Y$ in $Z$. Then the exceptional divisor $\widetilde{Y}=P^w(A)$, a weighted projective fibration over $Y$.

\begin{theorem}
Assume that $i_j< i_{j+1}  <i_{j+2}$ are three consecutive elements of $I$. Then the maps \bea  \diagram    & M_{(j,j+2)}(X) \dlto \drto \\  M_{(j,j+1)}(X) && M_{(j+1,j+2)}(X) \enddiagram   \eea
are weighted blow-ups.
The blow-up loci are
\bea  Y_{(j,j+1)}^- &=& [ (X_{{j+1}}^- \setminus X_{j+1}) / \CC^*] \hookrightarrow M_{(j,j+1)}(X) \mbox{ and } \\  Y_{(j+1,j+2)}^+ &=& [( X_{{j+1}}^+ \setminus X_{j+1}) / \CC^*] \hookrightarrow M_{(j+1,j+2)}(X), \eea
and the exceptional divisor $E_{(j,j+2)}=Y_{(j,j+1)}^-\times_{X}Y_{(j+1,j+2)}^+$ is a weighted projective fibration over $Y_{(j,j+1)}^-$ (as defined in \cite{noi5}):
\bea  E_{(j,j+2)} = [(A_{(j,j+1)}\setminus Y_{(j,j+1)}^-)/\CC^*],  \eea
 where $A_{(j,j+1)}$ is an affine bundle over $Y_{(j,j+1)}^-$ such that the normal bundle of its zero section $\cN_{Y_{(j,j+1)}^-|A_{(j,j+1)}}$ is obtained by descent from $\pi_{{j+1}}^{-*}\cN_{X_{{j+1}}|X_{{j+1}}^+}$ and the weights of the action on $\cN_{Y_{(j,j+1)}^-|A_{(j,j+1)}}$ are induced from the $\CC^*$--action on $X$.

The  coarse moduli spaces of $M_{(j,j+2)}(X)$ is the fiber product over the following variation of GIT diagram
\bea  \diagram     X^s_{(j,j+1)}//\CC^* \drto && X^s_{(j+1,j+2)}//\CC^* \dlto \\ &   (X^s_{(j,j+1)}\bigcup X_{{j+1}} \bigcup X^s_{(j+1,j+2)})//\CC^*. \enddiagram \eea

\end{theorem}

\begin{proof}
The map $M_{(j,j+2)}(X) \to M_{(j,j+1)}(X)$ is obtained by restricting the map between moduli spaces of weighted stable maps
$p: \overline{M}_{0, \cA(i_j,i_{j+2})}(X, \beta) \to \overline{M}_{0, \cA(i_j,i_{j+1})}(X, \beta)$ to those components of the fixed point loci satisfying conditions (1)-(2) in Definition \ref{M}. Indeed, these conditions are compatible with $p$ so $p(M_{(j,j+2)}(X))=M_{(j,j+1)}(X)$.

 Consider first the target $\PP^n$. We let $a_0=\frac{d-i_j}{d}$, $a_\infty=\frac{i_{j+2}}{d}$ and $a=\frac{1}{d}$. The boundary divisor of $\overline{M}_{0, \cA(i_j,i_{j+2})}(\PP^n, d,a)$ having nontrivial intersection with $M_{(j,j+2)}(\PP^n)$ is the image of the embedding \bea  \iota : \ol{M}_{0, (a_0, 1)}(\PP^n, i_{j+1}, a)  \times_{\PP^n}   \ol{M}_{0, (1, a_{\infty})}(\PP^n, d-i_{j+1}, a) \to \overline{M}_{0, \cA(i_j,i_{j+2})}(\PP^n, d,a).  \eea
We denote the non-trivial intersection of $\mbox{ Im } \iota$ with $M_{(j,j+2)}(X)$  by $E_{(j,j+2)}$. Its image in $M_{(j,j+1)}(X)$ is thus $Y_{(j,j+1)}^-:=M_{(j,j+1)}(X) \bigcap  \ol{M}_{0, (a_0, 1)}(\PP^n, i_{j+1}, a)$. Here the embedding
\bea  \ol{M}_{0, (a_0, 1)}(\PP^n, i_{j+1}, a) \hookrightarrow \overline{M}_{0, \cA(i_j,i_{j+1})}(\PP^n, d,a)\eea is given as follows: a tuple $(C, \{s_0, s_\infty\}, \cL, e)$ representing a point in $\ol{M}_{0, (a_0, 1)}(\PP^n, i_{j+1}, a)$ gives a tuple $(C, \{s_0, s_\infty\}, \cL\otimes \cO(d-i_{j+1}), e')$ where $e'$ is the composition of $e$ with $\cL \to \cL\otimes \cO(d-i_{j+1})$, and the last morphism has a zero with multiplicity $d-i_{j+1}$ at $s_\infty$. If $(C, \{s_0, s_\infty\}, \cL\otimes \cO(d-i_{j+1}), e')$ represents a point in $M_{(j,j+1)}(X)$, then  by equation (\ref{weights on sections}),
\bea \mu\circ \varphi (s_\infty)= d- \dim \Coker e_{s_{\infty}}\leq d-(d-i_{j+1})=i_{j+1}. \eea On the other hand, $[i_j, i_{j+1}] \subseteq \mbox{ Im } \mu\circ\varphi$ by Theorem \ref{consecutive}, hence $ \varphi(s_\infty)\in X_{{j+1}}$ and $\mbox{ Im } \varphi \subseteq X_{{j+1}}^-$ (as $C \cong \PP^1$).
 Furthermore, combining this with the proof of Theorem \ref{consecutive}, we obtain
\bea  Y_{(j,j+1)}^- \cong [(X_{{j+1}}^- \setminus X_{j+1}) /\CC^*].\eea

Similarly, we obtain that the image of $E_{(j,j+2)}$ in $M_{({j+1}, {j+2})}(X)$ is \bea Y_{(j+1,j+2)}^+ = M_{(j+1,j+2)}(X) \bigcap  \ol{M}_{0, (1, a_\infty)}(\PP^n, d-i_{j+1}, a) \cong [X_{{j+1}}^+ /\CC^*] \eea and $E_{(j,j+2)}=Y_{(j,j+1)}^-\times_{X_{{j+1}}}Y_{(j+1,j+2)}^+$.
Hence
\bea  E_{(j,j+2)} &\cong&  [(X_{{j+1}}^- \setminus X_{j+1})/\CC^*]\times_{X_{{j+1}}} [(X_{{j+1}}^+ \setminus X_{j+1})/\CC^*] \\
&\cong&  [(X_{{j+1}}^-\setminus X_{j+1}) \times_{X_{{j+1}}}[ (X_{{j+1}}^+\setminus X_{j+1}) /\CC^*] /\CC^*]. \eea
 $X_{{j+1}}^-\times_{X_{{j+1}}} [( X_{{j+1}}^+ \setminus X_{j+1})/\CC^*] \to [ ( X_{{j+1}}^+ \setminus X_{j+1})/\CC^*]$ is an affine bundle obtained as pull-back of  $X_{{j+1}}^-\to {X_{{j+1}}}$ thus the normal bundle of its zero section is $\pi^{+ *}_{{j+1}}(\cN_{X_{{j+1}}|X_{{j+1}}^-})$, and the weights of the $\CC^*$ action on $X_{{j+1}}^-\times_{X_{{j+1}}} [ X_{{j+1}}^+ /\CC^*]$ are those induced on $\cN_{X_{{j+1}}|X_{{j+1}}^-}$ by the action on $X$.
\end{proof}


\begin{theorem} \label{four_consecutive}
Assume that $i_j< i_{j+1}  < i_{j+2} <i_{j+3}$ are four consecutive elements of $I$. Then the following is a Cartesian diagram:
\bea   \diagram  & M_{(j,{j+3})}(X) \dlto \drto & \\      M_{(j,j+2)}(X)  \drto    & & M_{({j+1},{j+3})}(X)  \dlto \\ & M_{(j+1,j+2)}(X). &          \enddiagram  \eea
\end{theorem}
\begin{proof}
 We noted earlier the existence of morphisms $M_{j,k}(X) \to M_{j',k'}(X)$ whenever $j\leq j'$ and $k\geq k'$.
They are obtained by restriction from the birational morphisms between the moduli spaces of weighted stable maps.
Thus there exists a canonical map $a: M_{(j,{j+3})}(X)  \to M_{(j,j+2)}(X)  \times_{M_{(j+1,j+2)}(X)} M_{({j+1},{j+3})}(X)$.
To construct its inverse, consider an object in $M_{(j,j+2)}(X)  \times_{M_{(j+1,j+2)}(X)} M_{({j+1},{j+3})}(X)(B)$ over a scheme $B$.
To it there correspond tuples $(C_{(j,j+2)}, \{s_0, s_\infty\}, e_{(j,j+2)}:\cO_{C_{(j,j+2)}}^{n+1}\to \cL_{(j,j+2)})$ and $(C_{(j+1,j+3)}, \{s_0, s_\infty\}, e_{(j+1,j+3)}:\cO_{C_{(j+1,j+3)}}^{n+1}\to \cL_{(j+1,j+3)})$
together with a pair of isomorphism $(\sigma, \ol{\sigma})$:
\bea  \diagram C_{(j,j+2)}  \dto^{q^-} & C_{(j+1,j+3)} \dto^{q^+}   &   \mbox{ and }  & \cO_{C^-_{(j+1,j+2)}}^{n+1} \dto \rto^{e^-} & \cL^-_{(j+1,j+2)} \dto^{\cong \ol{\sigma}} \\ C^-_{(j+1,j+2)} \rto^{\sigma} & C^+_{(j+1,j+2)} & & \sigma^*\cO_{C^+_{(j+1,j+2)}}^{n+1} \rto^{\sigma^*e^+} & \sigma^*\cL^+_{(j+1,j+2)}.
\enddiagram  \eea such that \\
\noindent $\bullet$  The rational map $q^-$ contracts the locus $D_{0}$ where the line bundle \bea \omega_{C_{(j,j+2)}|B}(\frac{d-i_{j+1}}{d}s_0+ \frac{i_{j+2}}{d}s_\infty)\otimes \cL_{(j,j+2)}^a\eea fails to be ample, and similarly, $q^+$ the locus  $D_{\infty}$ where \bea \omega_{C_{(j+1,j+3)}|B}(\frac{d-i_{j+1}}{d}s_0+ \frac{i_{j+2}}{d}s_\infty)\otimes \cL_{(j+1,j+3)}^a\eea fails to be ample,
\\ \noindent $\bullet$ $\cL^-_{(j+1,j+2)}$ is obtained by descent from $\cL_{(j,j+2)}\otimes \cO_{C_{(j,j+2)}}(i_{j+1}D_0)$, and $\cL^+_{(j+1,j+2)}$ from $\cL_{(j+1,j+3)}\otimes \cO_{C_{(j+1,j+3)}}((d-i_{j+2})D_\infty)$,
\\ \noindent $\bullet$ The second diagram is commutative.

  We construct a flat family $C$ over $B$ by gluing $C_{(j,j+2)}\setminus\mbox{ Im } s_\infty$ and $C_{(j+1,j+3)}\setminus\mbox{ Im } s_0$ along the open subschemes  \bea \diagram C^-_{(j+1,j+2)}\setminus (\mbox{ Im } s_\infty \bigcup \mbox{ Im } s_0) \rto^{\cong } & C^+_{(j+1,j+2)}\setminus (\mbox{ Im } s_\infty \bigcup \mbox{ Im } s_0), \enddiagram \eea via the restriction of $\sigma$.
   By construction $C$ inherits the subschemes $D_0$ and $D_\infty$ from $ C_{(j,j+2)}$ and $C_{(j+1,j+3)}$ respectively,
  so their contractions give maps $p^+$ and $p^-$
  fitting into a commutative diagram
  \bea   \diagram  & C \dlto_{p^+} \drto^{p^-} & \\      C_{(j,j+2)} \drto_{q^-}    & & C_{(j+1,j+3)}  \dlto^{q^+} \\ & C_{(j+1,j+2)}. &          \enddiagram  \eea
  (as can be checked on the two patches).

  As well, the restrictions of $e_{(j,j+2)}$ and $e_{(j+1,j+3)}$ on $C_{(j,j+2)}\setminus\mbox{ Im } s_\infty$ and $C_{(j+1,j+3)}\setminus\mbox{ Im } s_0$ can be glued via the pullbacks of $\ol{\sigma}$ on ${C^-_{(j+1,j+2)}\setminus (\mbox{ Im } s_\infty \bigcup \mbox{ Im } s_0)}$ to form a morphism $e:\cO_{C}^{n+1}\to \cL$ of bundles over $C$.
  We have thus obtained a tuple $(C, \{s_0, s_\infty\}, e:\cO_{C}^{n+1}\to \cL)$ in $M_{(j,{j+3})}(X)(B)$. Indeed, the conditions in Definition \ref{M} are all inherited from the patches. To check that the stability properties in Definition \ref{weighted map spaces} are also inherited from those on $C_{(j,j+2)}$ and $C_{(j+1,j+3)}$, it is enough to note that:
   \\ \noindent $\bullet$ The sections $s_0$ and $s_\infty$ in $C$ are inherited from $ C_{(j,j+2)}$  and $C_{(j+1,j+3)}$, respectively.
  \\ \noindent $\bullet$ On each fibre $C_b$ over $b\in B$, the component $C_{0,b}$ containing $s_0(b)$ is also inherited from $ C_{(j,j+2)}$ so $\cL=\cL_{(j,j+2)}$ over $C_{0,b}$. Similarly $\cL=\cL_{(j+1,j+3)}$ over the component $C_{\infty, b}$ containing $s_\infty (b)$.
  \\ \noindent $\bullet$ By construction $\cL\otimes \cO_{C}(i_{j+1}D_0)$ descends to $\cL_{(j+1,j+3)}$ via $p^-$ and the composition
   \bea  \diagram  \cO_{C}^{n+1} \rto^{e} & \cL \rto &  \cL\otimes \cO_{C}(i_{j+1}D_0)  \enddiagram \eea
  descends to  $e_{(j+1,j+3)}$.  Thus for the points $b$ such that $C_{0,b}\subset D_0$ we have
   \bea \deg \cL_{C_{0,b}}= i_{j+1} = \dim \Coker e_{({j+1},{j+3}), s_0(b))} \eea
   and similarly,  $\deg \cL_{C_{\infty,b}}= d-i_{j+2}=\dim \Coker e_{(j,j+2), s_\infty(b))}$ whenever $C_{\infty,b} \subset D_\infty$.

    The canonical contractions $p^+$ and $p^-$ insure that the map  $M_{(j,j+2)}(X)  \times_{M_{(j+1,j+2)}(X)} M_{({j+1},{j+3})}(X) \to M_{(j,{j+3})}(X) $ thus constructed is the inverse of $a$.

\end{proof}

\begin{corollary} \label{1stpyramid}
For all the elements $0=i_0< i_1< i_2  < i_3 < ...< i_k=d$ of $I$, we obtain a net of Cartesian diagrams:

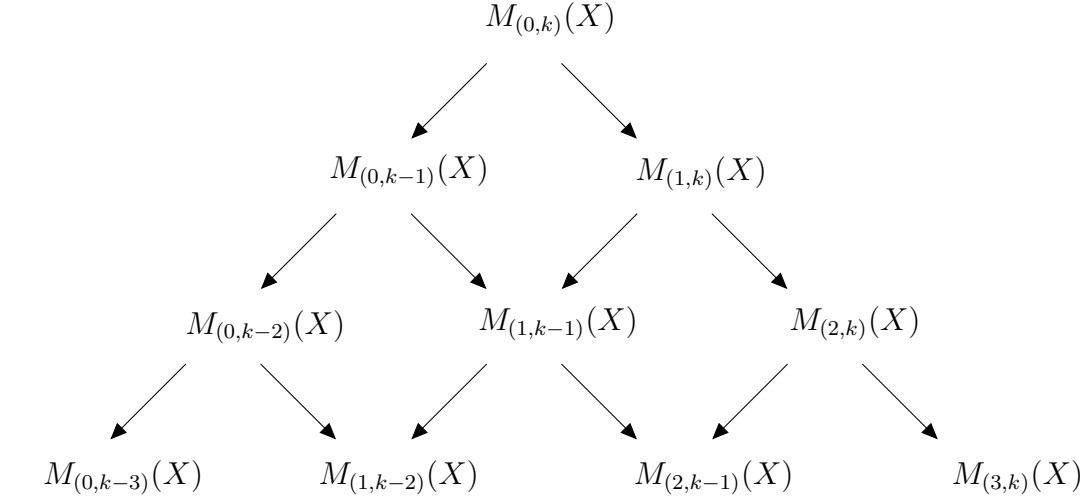
\begin{tikzpicture}[line cap=round,line join=round,>=triangle 45,x=.5cm,y=.5cm]
\clip(-3.12,-7.98) rectangle (25.77,5.75);
\draw (5.52,1.88) node[anchor=north west] {$M_{(0, k-1)}(X)$};
\draw (13.67,1.83) node[anchor=north west] {$M_{(1,k)}(X)$};
\draw (9.66,5.93) node[anchor=north west] {$M_{(0,k)}(X)$};
\draw (1.7,-2.26) node[anchor=north west] {$M_{(0,k-2)}(X)$};
\draw (9.48,-2.17) node[anchor=north west] {$M_{(1,k-1)}(X)$};
\draw (17.76,-2.17) node[anchor=north west] {$M_{(2,k)}(X)$};
\draw (-2.08,-6.27) node[anchor=north west] {$M_{(0,k-3)}(X)$};
\draw (5.25,-6.27) node[anchor=north west] {$M_{(1,k-2)}(X)$};
\draw (13.62,-6.27) node[anchor=north west] {$M_{(2,k-1)}(X)$};
\draw (22.08,-6.27) node[anchor=north west] {$M_{(3,k)}(X)$};
\draw [->] (12,4) -- (14,2);
\draw [->] (16,0) -- (18,-2);
\draw [->] (20,-4) -- (22,-6);
\draw [->] (10,4) -- (8,2);
\draw [->] (14,0) -- (12,-2);
\draw [->] (18,-4) -- (16,-6);
\draw [->] (6,0) -- (4,-2);
\draw [->] (8,0) -- (10,-2);
\draw [->] (4,-4) -- (6,-6);
\draw [->] (2,-4) -- (0,-6);
\draw [->] (12,-4) -- (14,-6);
\draw [->] (10,-4) -- (8,-6);
\draw (-3.08,-7.35) node[anchor=north west] {$..............................................................................................................................$};
\end{tikzpicture}

At the top, $M_{(0,k)}(X)=\underleftarrow{\lim}{M_{(j,j')}(X)}$ is a closed and open substack of $\ol{M}_{0,2}(X, \beta)^{\CC^*}$.
\end{corollary}

\begin{proof}
We note that the weights $\cA(0,d)=(1,1)$ so $M_{(0,k)}(X)$ is indeed a closed substack of $\ol{M}_{0,2}(X, \beta)^{\CC^*}$.
Conditions (1)-(2) in Definition \ref{M} are open in $\ol{M}_{0,2}(X, \beta)^{\CC^*}$. Indeed, consider any family $\pi: C\to B$ with sections $s_0$ and $s_\infty$ and stable map $\varphi: C \to X$, invariant under the $\CC^*$-action, and consider a fibre $C_b$ satisfying properties (1)-(2) in Definition \ref{M}. The nodes $s_0, s_1, ..., s_r=s_\infty$ of $C$ are mapped to points $x_j$ in the fixed point loci $X_{j}$ for some $j\in \{0,1,..., k\}$. Let $Z$ denote the set of indices $j$ with this property.

Since the components of the fixed locus of $X$ are closed and finitely many, there exists a neighborhood $B'$ of $b$, such that all the nodes in the fibres of $C_{|B'} \to B'$ are mapped into the same components $X_{j}$-s with $j\in Z$. (We used the continuity of $\varphi$  and properness of $\pi$.) Since $X_{j}^+/\CC^*$ and $X_{j}^-/\CC^*$ are compact and $\mbox{ Im } \varphi_b \bigcap X_{j}^\pm\not=\emptyset$, then $B'$ can be chosen such that whenever $C_{b'} \bigcap X_{j} \not=\emptyset$ for some $b'\in B'$, then  $C_{b'} \bigcap X^+_{j} \not=\emptyset$ and $C_{b'} \bigcap X^-_{j} \not=\emptyset$. Thus the preservation of flow is an open condition. As well, since there are finitely many ways to split a fibre according to the components of the fixed locus of $X$ where the nodes are mapped and according to the degrees of the components, it follows that fixing the degree of $\varphi$ on components of the fibres intersecting given components of the fixed point locus of $X$ gives an open condition as well.


\end{proof}

For each $j,j'\in \{0, 1, ..., k\}$, the universal family $\cU_{(j,j')}(X)$ over $M_{(j,j')}(X)$ admits a rational  evaluation map $\mbox{ev}_{(j,j')}$ into $X$ defined everywhere except on the images of the sections $s_0$ and $s_\infty$. The next step will be to replace $X$ by a more suitable target $X_{(j,j')}$.

\begin{notation} \label{notation_weights_for_targets}
We will work inside the weighted stable map spaces with three marked points $\ol{M}_{0,\cA(i, i', \Delta i)}(\PP^n, d, a)$,
where $0<a< \frac{1}{2d+1}$ and $\cA(i,i',\Delta i)$ represents the triple of weights
\bea  \begin{array}{lll}  a_\infty (i):= 1-(d-i-1)a, & a_0(i'):= 1-(i'-1)a,  & \mbox{ and } a_1(\Delta i):=(\Delta i-1)a,\end{array}\eea
for $-1\leq i < i' \leq d+1$ and $\Delta i \geq i'-i$.
\end{notation}

 In particular, $\ol{M}_{0,\cA(-1,d+1, d+2)}(\PP^n, d, a)\cong \PP^n_d:=\PP^{(n+1)(d+1)-1}$, is the "linear sigma model" representing trivial families of $\PP^1$-s with degree $d$ rational maps into $\PP^n$. We note that $\PP^n_d$ admits a natural $\CC^*$--action induced by the action with weights $(0,1)$ of $\CC^*$ on the source $\PP^1$.



Of special interest to us will be the cases when $\Delta i = i'-i$ and $0\leq i < i' \leq d$.

\begin{lemma} \label{2nd_action}
Each of the moduli spaces $\ol{M}_{0,\cA(i,i', i'-i)}(\PP^n, d, a)$  with  $0\leq i < i' \leq d$ parametrizes families of $\PP^1$-s with at least two disjoint sections and as such admits a natural $\CC^*$--action induced from the source.
\end{lemma}

\begin{proof}

Let $(\pi: C\to B, \{s_0, s_1, s_\infty\}, e: \cO^{n+1}_C \to \cL)$ be a family of $\cA(i,i', i'-i)$-weighted stable maps. We first note that all the fibres $C_b$ of the family $\pi: C \to B$ are $C_b \cong \PP^1$. Indeed, suppose $C_b$ is reducible and let $C'$ be a "tail" of $C_b$, namely a component such that $C_b\setminus C'$ is an open subset of $C_b$ with exactly one point on the boundary. As $a_1(i'-i)+da\leq 1$, condition (1) in Definition \ref{weighted map spaces} implies that $C'$ must contain $s_0(b)$ or $s_\infty(b)$.
If $C_0$ and $C_\infty$ are the components of $C_b$ containing $s_0(b)$ and $s_\infty(b)$ respectively, then we are in one of these three cases:
\begin{itemize}
\item[(i)]  $s_1(b)$ is in neither $C_0$ nor $C_\infty$. Then the choice of weights and condition (1) in Definition \ref{weighted map spaces} imply $ \deg \cL_{|C_0}\geq i'$ and $ \deg \cL_{|C_\infty}\geq d-i$, but $d-i+i'>d=\deg \cL$, which is impossible.
\item[(ii)] $s_1(b) \in C_0$. Then $\deg \cL_{|C_0}+\deg \cL_{|C_\infty}>i+(d-i)=d=\deg \cL$, which is impossible.
\item[(iii)] $s_1(b) \in C_\infty$. Then $\deg \cL_{|C_0}+\deg \cL_{|C_\infty}>i'+(d-i')=d=\deg \cL$, which is impossible.
\end{itemize}
Hence $C_0=C_\infty$.
Condition (3) in Definition \ref{weighted map spaces} insures that $s_0(b)$  and $s_\infty(b)$ are always distinct. Hence $C=\PP(\pi_*\cO_C(\mbox{ Im } s_0))$
which admits a natural  $\CC^*$--action with weights $(0, 1)$ on the fibres of $\pi$.
\end{proof}

\begin{remark} \label{chambers_for_weights}
We note that for $a_1:=(i'-i-1)a$, all values $a_0$ and $a_\infty$ with
\bea  \begin{array}{lll}  1-(d-i)a < a_\infty \leq  1-(d-i-1)a \mbox{ and } & 1-i'a < a_0\leq  1-(i'-1)a,\end{array}\eea
  define the same moduli space $\ol{M}_{0,\cA'(i,i', i'-i)}(\PP^n, d, a)$.
\end{remark}

\begin{lemma} \label{also2nd_action}
  For positive integers $0\leq i < i' < i''< i''' \leq d$, there is a commutative diagram of rational maps and morphisms as illustrated below, such that the lowest square is Cartesian. All the spaces in the diagram below  admit two $\CC^*$--actions induced from the target $X$ and source respectively, and the natural maps between them are  $\CC^*$--equivariant:

\end{lemma}

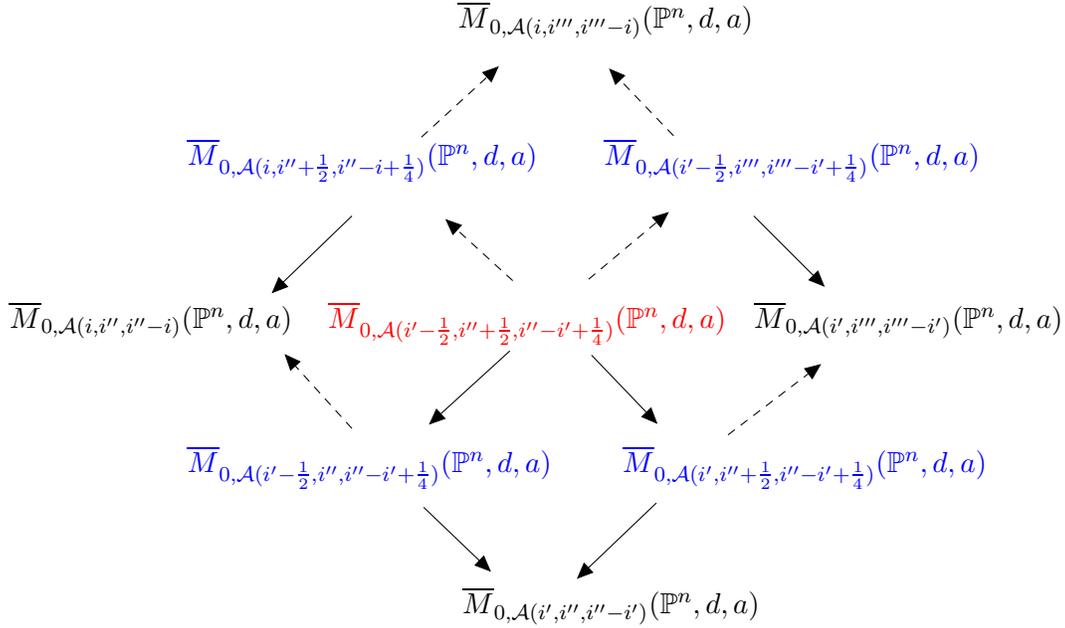
\definecolor{xdxdff}{rgb}{0.49,0.49,1}
\definecolor{ffqqqq}{rgb}{1,0,0}
\definecolor{qqqqff}{rgb}{0,0,1}
\begin{tikzpicture}[line cap=round,line join=round,>=triangle 45,x=.95cm,y=.95cm]
\clip(3,-3.7) rectangle (19.39,5.82);
\draw [color=qqqqff](5.5,3.6) node[anchor=north west] {\small{$\overline{M}_{0, \mathcal{A}(i,i''+\frac{1}{2},i''-i+\frac{1}{4})}{(\mathbb P^n, d, a)}$}};
\draw [color=qqqqff](11.33,3.6) node[anchor=north west] {\small{$\overline{M}_{0, \mathcal{A}(i'-\frac{1}{2},i''',i'''-i'+\frac{1}{4})}(\mathbb P^n, d, a)$}};
\draw (3.01,1.32) node[anchor=north west] {\small{$\overline{M}_{0, \mathcal{A}(i,i'',i''-i)}(\mathbb P^n, d, a)$}};
\draw [color=ffqqqq](7.47,1.32) node[anchor=north west] {\small{$\overline{M}_{0, \mathcal{A}(i'-\frac{1}{2},i''+\frac{1}{2},i''-i'+\frac{1}{4})}(\mathbb P^n, d, a)$}};
\draw (13.43,1.32) node[anchor=north west] {\small{$\overline{M}_{0, \mathcal{A}(i',i''',i'''-i')}(\mathbb P^n, d, a)$}};
\draw [color=qqqqff](5.5,-0.7) node[anchor=north west] {\small{$\overline{M}_{0, \mathcal{A}(i'-\frac{1}{2},i'',i''-i'+\frac{1}{4})}(\mathbb P^n, d, a)$}};
\draw (9.29,5.52) node[anchor=north west] {\small{$\overline{M}_{0, \mathcal{A}(i,i''',i'''-i)}(\mathbb P^n, d, a)$}};
\draw [color=qqqqff](11.59,-0.7) node[anchor=north west] {\small{$\overline{M}_{0, \mathcal{A}(i',i''+\frac{1}{2},i''-i'+\frac{1}{4})}(\mathbb P^n, d, a)$}};
\draw (9.35,-2.68) node[anchor=north west] {\small{$\overline{M}_{0, \mathcal{A}(i',i'',i''-i')}(\mathbb P^n, d, a)$}};
\draw [->,dash pattern=on 3pt off 3pt] (8.93,3.48) -- (10.01,4.47);
\draw [->,dash pattern=on 3pt off 3pt] (12.45,3.5) -- (11.55,4.44);
\draw [->] (7.95,2.37) -- (6.83,1.29);
\draw [->,dash pattern=on 3pt off 3pt] (11.27,1.49) -- (12.39,2.43);
\draw [->] (13.58,2.37) -- (14.57,1.38);
\draw [->,dash pattern=on 3pt off 3pt] (10.2,1.47) -- (9.26,2.33);
\draw [->] (11.31,0.42) -- (12.23,-0.54);
\draw [->] (10.16,0.48) -- (9.03,-0.55);
\draw [->] (12.21,-1.65) -- (11.1,-2.71);
\draw [->,dash pattern=on 3pt off 3pt] (13.22,-0.69) -- (14.52,0.3);
\draw [->,dash pattern=on 3pt off 3pt] (7.95,-0.6) -- (7.01,0.44);
\draw [->] (8.96,-1.71) -- (9.9,-2.61);

\end{tikzpicture}

\begin{proof}
We first note the existence of a birational morphism
\bea \overline{M}_{0, \mathcal{A}(i,i''+\frac{1}{2},i''-i+\frac{1}{4})}(\mathbb P^n, d, a) \to \overline{M}_{0, \mathcal{A}(i,i'',i''-i)}(\mathbb P^n, d, a). \eea
(Indeed, even if not all the weights defining the first space are larger than the corresponding weights of $\mathcal{A}(i,i'',i''-i)$, they are all larger than other triples in the same chamber as specified in Remark \ref{chambers_for_weights}.) The restriction of the map above to the exceptional locus is
\bea    \ol{M}_{0,(a_0(i''),1)}(\PP^n, i'', a) \times_{\PP^n} \ol{M}_{0,(1, a_1(i''-i), a_\infty(i))}(\PP^n, d-i'', a) \to \ol{M}_{0,(a_0(i''),1)}(\PP^n, i'', a), \eea
representing weighted stable maps of splitting type:

\definecolor{qqqqff}{rgb}{0,0,1}
\definecolor{xdxdff}{rgb}{0.49,0.49,1}
\begin{tikzpicture}[line cap=round,line join=round,>=triangle 45,x=.9cm,y=.9cm]
\clip(3,-1) rectangle (19,1);
\draw [shift={(8.44,-2.36)}] plot[domain=0.99:2.17,variable=\t]({1*2.83*cos(\t r)+0*2.83*sin(\t r)},{0*2.83*cos(\t r)+1*2.83*sin(\t r)});
\draw [shift={(5.34,-1.76)}] plot[domain=0.86:2.29,variable=\t]({1*2.31*cos(\t r)+0*2.31*sin(\t r)},{0*2.31*cos(\t r)+1*2.31*sin(\t r)});
\draw [shift={(15,-2)}] plot[domain=0.92:2.21,variable=\t]({1*2.51*cos(\t r)+0*2.51*sin(\t r)},{0*2.51*cos(\t r)+1*2.51*sin(\t r)});
\draw (3.66,0.78) node[anchor=north west] {$s_0$};
\draw (8.28,1.14) node[anchor=north west] {$s_1$};
\draw (9.5,0.66) node[anchor=north west] {$s_\infty$};
\draw (13.46,0.8) node[anchor=north west] {$s_0$};
\draw (16.24,0.76) node[anchor=north west] {$s_1=s_\infty$};
\draw (4.84,0.3) node[anchor=north west] {$\deg i''$};
\draw (7.,0.3) node[anchor=north west] {$\deg (d-i'')$};
\draw (14,0.3) node[anchor=north west] {$\deg i''$};
\draw (16.1,0.2) node[anchor=north west] {{$m_\infty=d-i''$}};
\draw [->] (10.76,0.16) -- (12.94,0.14);
\begin{scriptsize}
\fill [color=xdxdff] (10,0) circle (1.5pt);
\fill [color=qqqqff] (3.8,0) circle (1.5pt);
\fill [color=xdxdff] (8.38,0.47) circle (1.5pt);
\fill [color=xdxdff] (16.52,0) circle (1.5pt);
\fill [color=xdxdff] (13.5,0) circle (1.5pt);
\end{scriptsize}
\end{tikzpicture}

(We note that the fractionary parts of the weights play no significant role in the moduli problem for
 the exceptional divisor and can be omitted here.)

Indeed, redoing the casework from the proof of Lemma \ref{2nd_action} in the case of the moduli space $\overline{M}_{0, \mathcal{A}(i,i''+\frac{1}{2},i''-i+\frac{1}{4})}(\mathbb P^n, d, a)$:
 \begin{itemize}
 \item If  $s_1(b) \in C_\infty$, then
\bea  \deg \cL_{|C_0} > i''-\frac{1}{2} \mbox{ and } \deg \cL_{|C_\infty}  > d-i''-\frac{1}{4},\eea
yielding $\cL_{|C_0} = i''$ and $\deg \cL_{|C_\infty}  = d-i''$ as in the exceptional divisor above.
\item Whereas if $s_1(b) \in C_0$, then
\bea  \deg \cL_{|C_0} > \frac{1}{4}+i \mbox{ and } \deg \cL_{|C_\infty} >  d-i-\frac{1}{4}\eea
which is impossible.
\end{itemize}

The blow-up locus $\ol{M}_{0,(a_0(i''),1)}(\PP^n, i'', a)$  represents objects in
 $\overline{M}_{0, \mathcal{A}(i,i'',i''-i)}(\mathbb P^n, d, a)$ for which $s_1=s_\infty$ and $\dim \Coker e_{s_1} = d-i''$.
 This locus is fixed by the $\CC^*$--action defined in Lemma \ref{2nd_action} and so there is an induced $\CC^*$--action on the weighted blow-up
 $\overline{M}_{0, \mathcal{A}(i,i''+\frac{1}{2},i''-i+\frac{1}{4})}(\mathbb P^n, d, a)$.

\bigskip

   The two other SW--oriented arrows are blow-downs whose exceptional loci represent the same splitting type as above. Similarly, the SE--oriented arrows are blow-downs which over the exceptional locus represent contractions of the following splitting type within their moduli spaces:

   \definecolor{qqqqff}{rgb}{0,0,1}
\definecolor{xdxdff}{rgb}{0.49,0.49,1}
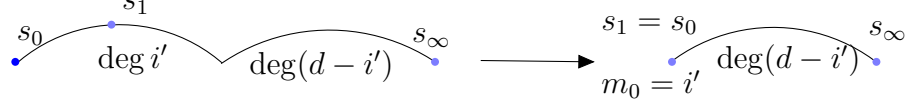
\begin{tikzpicture}[line cap=round,line join=round,>=triangle 45,x=.9cm,y=.9cm]
\clip(3,-1) rectangle (18.5,1);
\draw [shift={(8.44,-2.36)}] plot[domain=0.99:2.17,variable=\t]({1*2.83*cos(\t r)+0*2.83*sin(\t r)},{0*2.83*cos(\t r)+1*2.83*sin(\t r)});
\draw [shift={(5.34,-1.76)}] plot[domain=0.86:2.29,variable=\t]({1*2.31*cos(\t r)+0*2.31*sin(\t r)},{0*2.31*cos(\t r)+1*2.31*sin(\t r)});
\draw [shift={(15,-2)}] plot[domain=0.92:2.21,variable=\t]({1*2.51*cos(\t r)+0*2.51*sin(\t r)},{0*2.51*cos(\t r)+1*2.51*sin(\t r)});
\draw (3.66,0.68) node[anchor=north west] {$s_0$};
\draw (9.5,0.66) node[anchor=north west] {$s_\infty$};
\draw (12.3,0.84) node[anchor=north west] {$s_1=s_0$};
\draw (16.24,0.76) node[anchor=north west] {$s_\infty$};
\draw (4.86,0.42) node[anchor=north west] {$\deg i'$};
\draw (7.1,0.34) node[anchor=north west] {$\deg (d-i')$};
\draw (14,0.42) node[anchor=north west] {$\deg (d-i')$};
\draw (12.3,0.04) node[anchor=north west] {$m_0=i'$};
\draw [->] (10.68,0.02) -- (12.36,0);
\draw (5.22,1.14) node[anchor=north west] {$s_1$};
\begin{scriptsize}
\fill [color=xdxdff] (10,0) circle (1.5pt);
\fill [color=qqqqff] (3.8,0) circle (1.5pt);
\fill [color=xdxdff] (16.52,0) circle (1.5pt);
\fill [color=xdxdff] (13.5,0) circle (1.5pt);
\fill [color=xdxdff] (5.22,0.55) circle (1.5pt);
\end{scriptsize}
\end{tikzpicture}

  As the images of the sections $s_0$ and $s_\infty$ are always disjoint, then so are the two different blow-up loci in  $\overline{M}_{0, \mathcal{A}(i',i'',i''-i')}(\mathbb P^n, d, a)$, and so the result of these two successive blow-ups $\overline{M}_{0, \mathcal{A}(i'-\frac{1}{2},i''+\frac{1}{2},i''-i'+\frac{1}{4})}(\mathbb P^n, d, a)$ is the Cartesian product of the two individual blow-ups along those two loci. As noted earlier, the blow-up loci are fixed by the $\CC^*$--action induced from the source and so  $\overline{M}_{0, \mathcal{A}(i',i'',i''-i')}(\mathbb P^n, d, a)$ also inherits such a $\CC^*$--action making all the morphisms in the diagram equivariant.

  The NE--oriented rational map with target $\overline{M}_{0, \mathcal{A}(i,i''',i'''-i)}(\mathbb P^n, d, a)$ can be understood as obtained by successively blow-ing up the loci for which $ i''\leq m_0 <i'''$ (and their strict transforms), in decreasing order of $m_0$, following by blow-downs of all exceptional divisors which had resolved the cases $m_0>i''$. Similarly for the NW--oriented maps and
   $ d-i' \leq m_\infty < d-i$. These loci are invariant under the $\CC^*$--actions so the rational maps are equivariant where defined. The same arguments are valid for the remaining dotted arrows in the diagram.
\end{proof}

\begin{definition} \label{notation_targets}

The two $\CC^*$--actions induced from source and target respectively on the moduli spaces in Lemma \ref{also2nd_action} together form an action of $\CC^* \times \CC^*$ on these spaces. Consider the inverse diagonal action of $\CC^*$ given by the
embedding  $\CC^*\to \CC^*\times \CC^*$ where $t \to (t, t^{-1})$.

For each pair of positive integers $(j,j')$ with $0<j<j'<k$ and their corresponding wall elements $i_j, i_{j'} \in I$ in the image of the moment map of $X$, we define $X_{(j, j')}$ to be the subspace of $(\overline{M}_{0, \mathcal{A}(i_j,i_{j'},i_{j'}-i_j)}(\mathbb P^n, d, a))^{\CC^*}$ whose points parametrize $\CC^*$-fixed weighted stable maps to $X$ which are both flow-preserving and action-class-preserving, in the sense of Definition \ref{M}. Similarly to the embedding
\bea X_{(j, j')}  \hookrightarrow \overline{M}_{0, \mathcal{A}(i_j,i_{j'},i_{j'}-i_j)}(\mathbb P^n, d, a)^{\CC^*}, \eea
 we also define the following:  \bea & & X_{(j, j'+\frac{1}{2})}  \hookrightarrow  \overline{M}_{0, \mathcal{A}(i_j,i_{j'}+\frac{1}{2},i_{j'}-i_j+\frac{1}{4})}{(\mathbb P^n, d, a)}^{\CC^*},  \\
& & X_{(j-\frac{1}{2}, j')} \hookrightarrow  \overline{M}_{0, \mathcal{A}(i_j-\frac{1}{2},i_{j'},i_{j'}-i_j+\frac{1}{4})}{(\mathbb P^n, d, a)}^{\CC^*}, \mbox{ and } \\
& &X_{(j-\frac{1}{2}, j'+\frac{1}{2})} \hookrightarrow \overline{M}_{0, \mathcal{A}(i_j-\frac{1}{2},i_{j'}+\frac{1}{2},i_{j'}-i_j+\frac{1}{4})}{(\mathbb P^n, d, a)}^{\CC^*}.  \eea

\end{definition}

\begin{corollary} \label{DiamondX}
Let $j, j'$ be two integers with $0<j<j+1<j'-1<j'<k$ and consider the following walls in $I$:
\bea   i=i_j \mbox{, } i'=i_{j+1}  \mbox{, } i''=i_{j'-1} \mbox{ and  }  i'''=i_{j'}. \eea
The diagram in Lemma \ref{also2nd_action} restricted to the flow-preserving, action-class-preserving part of the fixed locus yields the following four Cartesian squares:

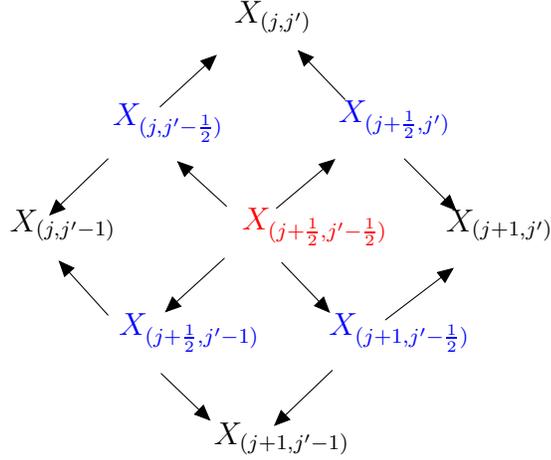
\begin{center} \begin{tikzpicture}[line cap=round,line join=round,>=triangle 45,x=.7cm,y=.7cm]
\clip(4.59,-3.78) rectangle (16.92,5.85);
\draw [color=qqqqff](7.83,3.63) node[anchor=north west] {$X_{(j,j'-\frac{1}{2})}$};
\draw [color=qqqqff](12.12,3.69) node[anchor=north west] {$X_{(j+\frac{1}{2}, j')}$};
\draw (5.88,1.62) node[anchor=north west] {$X_{(j, j'-1)}$};
\draw [color=ffqqqq](10.29,1.65) node[anchor=north west] {$X_{(j+\frac{1}{2}, j'-\frac{1}{2})}$};
\draw (14.16,1.62) node[anchor=north west] {$X_{(j+1,j')}$};
\draw [color=qqqqff](7.95,-0.36) node[anchor=north west] {$X_{(j+\frac{1}{2}, j'-1)}$};
\draw (10.14,5.58) node[anchor=north west] {$X_{(j,j')}$};
\draw [color=qqqqff](11.94,-0.36) node[anchor=north west] {$X_{(j+1,j'-\frac{1}{2})}$};
\draw (9.75,-2.43) node[anchor=north west] {$X_{(j+1,j'-1)}$};
\draw [->] (8.94,3.36) -- (10.02,4.35);
\draw [->] (12.45,3.5) -- (11.55,4.44);
\draw [->] (7.95,2.37) -- (6.83,1.29);
\draw [->] (11.15,1.43) -- (12.27,2.37);
\draw [->] (13.58,2.37) -- (14.57,1.38);
\draw [->] (10.2,1.47) -- (9.26,2.33);
\draw [->] (11.25,0.4) -- (12.17,-0.56);
\draw [->] (10.16,0.48) -- (9.03,-0.55);
\draw [->] (12.21,-1.65) -- (11.1,-2.71);
\draw [->] (13.22,-0.69) -- (14.52,0.3);
\draw [->] (7.95,-0.6) -- (7.01,0.44);
\draw [->] (8.96,-1.71) -- (9.9,-2.61);
\end{tikzpicture}
\end{center}
\end{corollary}

\begin{proof}
The conditions of flow and action-class preservations are compatible with the maps in Lemma \ref{also2nd_action} leading to the diagram above. We check that all the arrows in this diagram represent well defined morphism.

Indeed,  the points where
$\overline{M}_{0, \mathcal{A}(i,i''+\frac{1}{2},i''-i+\frac{1}{4})}(\mathbb P^n, d, a) \dashrightarrow\overline{M}_{0, \mathcal{A}(i,i''',i'''-i)}(\mathbb P^n, d, a)$ is not well defined correspond to the condition
$i''< m_0 <i'''$ on the moduli problem. This locus intersects $X_{(j, j'-\frac{1}{2})}$ trivially
since $i''=i_{j'-1}$ and $i'''=i_{j'}$ are consecutive walls.
  Hence $X_{(j, j'-\frac{1}{2})} \to X_{(j, j')}$ is everywhere well-defined, and similarly for all the other maps in the diagram.

 We denote by $Z_{(j, j'-\frac{1}{2})}$ and $Z_{(j+\frac{1}{2}, j')}$ the images of the exceptional loci of the two maps to $X_{(j, j')}$. Note that $Z_{(j, j'-\frac{1}{2})}$ and $Z_{(j+\frac{1}{2}, j')}$  correspond to conditions $m_0=i_{j'-1}$ and $m_\infty=d-i_{j+1}$ respectively, which cannot be satisfied simultaneously since $i_{j'-1}>i_{j+1}$.  Hence the upper square in the diagram is Cartesian, as $X_{(j+\frac{1}{2}, j'-\frac{1}{2})}$ can be thought of as constructed by gluing $X_{(j, j'-\frac{1}{2})}$ and $X_{(j+\frac{1}{2}, j')}$ along $X_{(j, j')}\setminus (Z_{(j, j'-\frac{1}{2})} \bigcup Z_{(j+\frac{1}{2}, j')})$.
Similar arguments show that all the other squares in the diagram are Cartesian.

\end{proof}

\begin{lemma} \label{action_on_X0k}
There is a $\CC^*$-equivariant map $X_{(0,k)} \to X$  such that $X_{(0,k)}$ is the weighted blow-up of $X$ along the source and the sink components of the fixed locus.

\end{lemma}

\begin{proof}
For each point $x\in X$, the action-map $\varphi_x: \CC^* \to X$, given by $t\to t\cdot x$, can be represented by a point in $\PP^n_d$. We can define this explicitly as follows:  Let $d_0, ..., d_n$ be the weights of the $\CC^*$ action on $\PP^n$ for which  the embedding $X \hookrightarrow \PP^n$ is $\CC^*$--equivariant. Then we have an equivariant map
\bea  X \hookrightarrow \PP^n & \hookrightarrow & \PP^n_d, \\
 \left[(x_i)_{i\in \{0,...,n\} }\right] & \to & \left[(y_i^l)_{i\in \{0,...,n\}, l\in \{0,...,n\} }\right] \eea
where $y_i^l=\left\{ \begin{array}{l} x_i \mbox{ if } l=d_i, \\ 0 \mbox{ otherwise. } \end{array}\right.$

We think of $\PP^n_d$ as a space of weighted stable maps with 3 marked points $\ol{M}_{\cA}(\PP^n, d, a)$ with $0<a<\frac{1}{2d}$ and $0<a_i<1-da$ as well as $a_i+a_l>1$ for all $i,l\in \{0, \infty, 1\}$. Then with the embedding above, $\PP^n$ is the flow and action-class--preserving part of the fixed locus of $\PP^n_d$, for the inverse-diagonal action introduced earlier, and $X$ plays the similar role for $\ol{M}_{\cA}(X, \beta, a)$.

In the choice of weights for $\PP^n_d$, we can choose $a_1$ small enough to justify the existence of a $\CC^*$--equivariant birational morphism
\bea  \ol{M}_{\cA(0, d, d)}(\PP^n, d, a) \to \PP^n_d, \eea
Restricting to  $\ol{M}_{\cA(0, d, d)}(X, \beta, a)$, and then to the flow and action-class preserving parts of the fixed locus, we obtain a commutative diagram
\bea   \diagram  X_{(0,k)} \rto \dto & X \dto \\ \PP^n_{(0,k)} \rto & \PP^n    \enddiagram \eea
of $\CC^*$-equivariant morphisms. The action of $\CC^*$ on the moduli spaces is  $X_{(0,k)}$ and $\PP^n_{(0,k)}$ is induced from the target, which in this case is the same as the one induced from the source of the weighted stable maps.

The morphism $\PP^n_{(0,k)} \to  \PP^n $ can be described very explicitly in coordinates as the weighted blow-up of $\PP^n$ along the source and the sink of $\PP^n$, and similarly for $X_{(0,k)} \to X$. Indeed, the two exceptional divisors in  $X_{(0,k)}$ parametrize  orbits in $X$  with $s_1=s_0\in X_0$ and  $s_1=s_\infty \in X_\infty$ respectively, while the map $X_{(0,k)} \to X$ is exactly the evaluation map at $s_1$.
Thus the exceptional divisors are exactly the weighted projective fibrations
\bea  [(X_0^+\setminus X_0)/\CC^*] \to X_0 \mbox{ and } [(X_\infty^-\setminus X_\infty)/\CC^*] \to X_\infty. \eea
\end{proof}

From now on, unless explicitly stated otherwise, all spaces of weighted stable maps and their subspaces will be considered with the $\CC^*$--action induced from the action on the source.

\begin{proposition} \label{evaluation_maps}
Let $j, j'$ be positive integers such that $0<j<j'<k$ and let  $\cU_{(j,j')}(X)$ denote the universal family over $M_{(j,j')}(X)$, with evaluation map \bea \mbox{ev}_{(j,j')}: \cU_{(j,j')}(X) \dashrightarrow X \eea well defined everywhere except on the images of $s_0$ and $s_\infty$. The following properties hold for $X_{(j,j')}$:
\begin{itemize}
\item[a)] There is a well defined, $\CC^*$--equivariant map $\varphi_{(j,j')} : \cU_{(j,j')}(X) \to X_{(j,j')}$.
\item[b)] There is a $\CC^*$--equivariant birational map $X_{(j,j')} \dashrightarrow X$ defined everywhere outside the source and sink of $X_{(j,j')}$. Where defined, this map is an isomorphism on its image.
  \item[c)] Via the above isomorphism, the restriction of $\varphi_{(j,j')}$ to $\cU_{(j,j')}\setminus (\mbox{ Im } s_0 \bigcup \mbox{ Im } s_\infty )$ is identified with $\mbox{ev}_{(j,j')}$.
\end{itemize}
\end{proposition}

\begin{proof}
a) For any positive real number $a$ such that $0\leq a\leq \frac{1}{d}$ we have \bea  \cU_{(j,j')}(X) \hookrightarrow \ol{M}_{\cA^a(i_j, i_{j'})}(\PP^n, d, a)\eea where $\cA^a(i_j, i_{j'})$ can be chosen to be any triple $(a_0, a_\infty, a_1)$ satisfying
\bea  a_0 \in (1-a(i_j+1), 1-ai_j]\mbox{, } a_\infty \in (1-a(d-i_{j'}+1), 1-a(d-i_{j'})]\mbox{, } 0\leq a_1 <1-da, \eea
while for  $0\leq a< \frac{1}{2d}$  we have $X_{(j,j')} \hookrightarrow \ol{M}_{\cA(i_j, i_j', i_j'-i_j)}(\PP^n, d, a)$ with
\bea  a_0 \in (1-ai_{j'}, 1-a(i_{j'}-1)], a_\infty \in (1-a(d-i_{j}-1), 1-a(d-i_{j})],  a_1=a(i_{j'}-i_j-1). \eea
As $i_j+1 \leq i_{j'}$ and by choice $a<\frac{1}{2d}$, the weights in $\cA^a(i_j, i_{j'})$ are larger than those in $\cA(i_j, i_j', i_j'-i_j)$, justifying the existence of a birational morphism
\bea \ol{M}_{\cA^a(i_j, i_{j'})}(\PP^n, d, a) \to  \ol{M}_{\cA(i_j, i_j', i_j'-i_j)}(\PP^n, d, a).  \eea
Restriction to the flow and action-class--preserving part of the fixed point locus in the moduli spaces of target $X$ leads to a morphism $\varphi_{(j,j')} : \cU_{(j,j')}(X) \to X_{(j,j')}$.

At the level of moduli problems, the morphism $\varphi_{(j,j')} : \cU_{(j,j')}(X) \to X_{(j,j')}$  consists in contracting each chain of $\PP^1$-s with marked point $s_1$ down to the component containing $s_1$, while adding the degrees of the contracted components to the new multiplicities of $s_0$ and $s_\infty$.

b) The components of the fixed point locus of $X_{(j,j')}$ are the following:
\begin{itemize}
\item The source and sink, defined by the conditions $s_0=s_1$ and $s_1=s_\infty$, respectively.
\item The remaining loci, parametrizing $\PP^1$-s together with contractions to a fixed point of $X$, with $m_0=i$ and $m_\infty=d-i$ for some $i\in I$.
\end{itemize}

We have already illustrated the existence of a chain of birational maps between  $X_{(j,j')}$ and $X_{(0,k)}$, so that the composition $X_{(j,j')} \dashrightarrow X_{(0,k)}$ is  defined everywhere except some subspaces of the source and sink. Indeed, in the diagram at Corollary \ref{DiamondX}, all rational maps which invert the NE and NW--oriented arrows have this property.

Finally, Lemma \ref{action_on_X0k} shows that $X_{(0,k)}$ is the weighted blow-up of $X$ along the source and the sink components of the fixed locus.

 c) Outside of the source and sink, we can identify each point $b$ of $X_{(j,j')}$ with $s_1(b) \in X$. Indeed, this map can be inverted by identifying $s_1(b)$ with its action-map $t\to t\cdot s_1(b)$. With this identification, the map $\varphi_{(j,j')}$ described in part (a) of the proof is identified with $\mbox{ev}_{(j,j')}$ on $\cU_{(j,j')}\setminus (\mbox{ Im } s_0 \bigcup \mbox{ Im } s_\infty )$.
\end{proof}

\begin{remark} \label{identification}
Let $(X_{(j,j')})_0$ and $(X_{(j,j')})_\infty$  denote the source and sink of $X_{(j,j')}$. Via the birational map $X_{(j,j')} \dashrightarrow X$, we can identify the open subspace  $X_{(j,j')}\setminus (X_{(j,j')})_0\bigcup (X_{(j,j')})_\infty)$ with $X_{(j,j')}^o:=(\bigcup_{l'\leq j' }X_{l'}^+) \bigcap (\bigcup_{l\geq j} X_{l}^- )$. Then  $X_{(j,j')}=\overline{X^o}_{(j,j')}$ is a compactification of this compatible with the $\CC^*$--action, constructed like in the first part of the proof of Theorem \ref{consecutive}. Indeed, the source and sink of $X_{(j,j')}$ are birational to the source and sink in $X_{(0,k)}$, and hence divisors (see Lemma \ref{action_on_X0k}). Hence $(X_{(j,j')})_0^+ \to (X_{(j,j')})_0$ is an affine fibration with  $\AA^1$ fibres, and similarly for $(X_{(j,j')})_\infty^- \to (X_{(j,j')})_\infty$.

\end{remark}

Next we describe the maps between the target spaces  $X_{(j,j')}$ in terms of the $\CC^*$--action on $X$.

\begin{notation}
We denote by $\overline{X_{j'}^+}$ the closure of ${X_{j'}^+}$ in $X_{(j,j'+1)}$. As an alternate definition, \bea  \overline{X_{j'}^+}\cong [\left( (X_{j'}^{+}\times \AA^1) \setminus (X_{j'}\times \{0\})\right)/\CC^*], \eea  up to isomorphisms the unique compactification of $X_{j'}^{+}$ to a projective fibration
over $X_{j'}$.
Indeed,
$\overline{X_{j'}^+}\setminus X_{j'} \subset (X_{(j,j'+1)})_\infty^-$, which is an affine fibration with one-dimensional fibres over $(X_{(j,j'+1)})_\infty$, and so $\overline{X_{j'}^+}\setminus X_{j'}$ is an $\AA^1$-fibration over its image in $(X_{(j,j'+1)})_\infty$. This image thus forms the divisor at infinity of  $\overline{X_{j'}^+} \to X_{j'}$.
\end{notation}

Recall the notation $Y^-_{(j'-1, j')}=[( X^-_{j'}-X_{j'}) /\CC^*]$.

\begin{theorem}
a) The exceptional locus of the two birational morphisms \bea  X_{(j,j')} \longleftarrow  X_{(j,j'+\frac{1}{2})} \longrightarrow  X_{(j,j'+1)} \eea is $E:= Y^-_{(j'-1, j')} \times_{X_{j'}} \overline{X_{j'}^+}$. The restrictions of the two morphisms to the exceptional locus give the top two maps in the Cartesian diagram
\bea \diagram  &&  Y^-_{(j'-1, j')} \times_{X_{j'}} \overline{X_{j'}^+} \dllto_{\ol{g_{j'}^+}} \drrto^{f_{j'}^-} && \\
Y^-_{(j'-1, j')} \drrto^{p_{j'}^{-}}  &&&& \overline{X_{j'}^+} \dllto_{\overline{q^+_{j'}}} \\ &&{X_{j'}}.&&
 \enddiagram \eea

 Both $p_{j'}^{-}$ and $\overline{q^+_{j'}}$ are weighted projective fibrations over $X_{j'}$.

b) Both morphisms $ X_{(j,j')} \longleftarrow  X_{(j,j'+\frac{1}{2})} \longrightarrow  X_{(j,j'+1)} $ are weighted blow-ups along the smooth loci $Y^-_{(j'-1, j')}$ and $\overline{X_{j'}^+}$, respectively, with the weights induced from the $\CC^*$--action. All spaces are smooth Deligne-Mumford stacks.

c) Let $\cF$ denote the conormal bundle of $E$ in $X_{(j,j'+\frac{1}{2})}$, and let $\pi$ denote the morphism $X_{(j,j'+\frac{1}{2})}\to X_{(j,j'+1)}$. The associated affine fibration over the blow-up locus $\overline{X_{j'}^+}$, defined as $A=\Spec (\oplus_{n\geq 0} \pi_* \cF^n)$, satisfies the property
\bea A\otimes\cO_{\ol{X_{j'}^+}}(-1):= \Spec (\oplus_{n\geq 0} \pi_* \cF^n \otimes\cO_{\ol{X_{j'}^+}}(-n)) \cong X^-_{j'} \times_{X_{j'}}\overline{X_{j'}^+}, \eea
with the natural $\CC^*$--action induced from the action on $X$.
The normal bundle of the zero section in  $A$ is  $\overline{q^+_{j'}}^*\cN_{X_{j'}|X}^-\otimes \cO_{\ol{X_{j'}^+}}(1)$.

Similarly, if $B \to Y^-_{(j'-1, j')}$ is the affine fibration associated to the weighted blow-up $X_{(j,j'+\frac{1}{2})} \to X_{(j,j')} $, then
\bea B\otimes\cO_{Y^-_{(j'-1, j')}}(-1) \cong Y^-_{(j'-1, j')}\times_{X_{j'}}(X^+_{j'}\oplus \cO_{X_j'}).\eea
The normal bundle of the zero section in $B$ is $p_{j'}^{- *}(\cN_{X_{j'}|X}^+\oplus \cO_{X_{j'}})\otimes \cO_{Y^-_{(j'-1, j')}}(1)$.

\end{theorem}

\begin{proof}

a)  Following the proof of Lemma \ref{also2nd_action}, the image of the exceptional locus of $ X_{(j,j'+\frac{1}{2})} \rightarrow  X_{(j,j'+1)}$ represents objects in $X_{(j,j'+1)}$ for which $\dim \Coker e_{s_0} = i_{j'}$. Via the identifications in Remark \ref{identification}, removing the sink from this locus yields $X^+_{j'}$, so the locus is exactly $\overline{X_{j'}^+}$, a weighted projective fibration over $X_{j'}$.

   On the other hand, the image of the exceptional locus of $ X_{(j,j'+\frac{1}{2})} \to  X_{(j,j')}$ represents objects in $X_{(j,j')}$  for which $s_1=s_\infty$ and $\dim \Coker e_{s_\infty} = d-i_{j'}$. This can be identified with the sublocus of $M_{(j', j'+1)}(X)$ parametrizing orbits in $X_{{j'}}^-$. We know $M_{(j', j'+1)}(X) \cong [X^s_{(j', j'+1)}/\CC^*]$ by Theorem \ref{consecutive} and by the same arguments, the blow-up locus is identified with $[(X_{{j'}}^-\setminus X_{j'})/\CC^*]$, which is a weighted projective fibration over $X_{{j'}}$.

   The shared exceptional locus for $X_{(j,j')}\leftarrow X_{(j,j'+\frac{1}{2})} \rightarrow  X_{(j,j'+1)}$ is thus \bea [(X_{{j'}}^-\setminus X_{j'}) /\CC^*] \times_{X_{j'}} \overline{X_{j'}^+}. \eea
This is a weighted projective fibration over the loci $[ (X_{j'}^-\setminus X_{j'})/\CC^*]$ and $\overline{X_{j'}^+}$; all three are smooth Deligne-Mumford stacks.

b) Recall the embedding $X_{(j, j'+1)} \hookrightarrow (\overline{M}_{0,\mathcal{A}(i_j,i_{j'+1},i_{j'+1}-i_j)}(\mathbb P^n, d, a))^{\CC^*}$. For simplicity, we will denote by $M$ the complement in $\overline{M}_{0,\mathcal{A}(i_j,i_{j'+1},i_{j'+1}-i_j)}(\mathbb P^n, d,a))^{\CC^*}$ of the locus made by stable map spaces with $m_0>i_{j'}$. Similarly, let
$\widetilde{M}$ denote the preimage of $M$ in $\overline{M}_{0, \mathcal{A}(i_j,i_{j'}+\frac{1}{2},i_{j'}-i_j+\frac{1}{4})}{(\mathbb P^n, d, a)}^{\CC^*}$. Then as remarked in Corollary \ref{DiamondX}, $X_{(j, j'+1)} \hookrightarrow M$ and $X_{(j, j'+\frac{1}{2})}\hookrightarrow \widetilde{M}$.


We know that $\widetilde{M}\to M$ is a $\CC^*$--equivariant weighted blow-up and the blow-up locus parametrizes weighted stable maps for which $m_0=i_{j'}$. Thus in terms of the $\CC^*$--action, the blow-up locus can be described as $\overline{M_{j'}^+}$, where
\begin{itemize}
\item $M_{j'}$ denotes the $\CC^*$--fixed point locus parametrizing weighted stable maps for which $m_0=i_{j'}$ and $m_\infty=d-i_{j'}$, the source being contracted to a fixed point of $\PP^n$.
\item $M_{j'}^+:=\{ c\in M\mbox{; } \lim_{t\to 0} t\cdot c\in M_{j'}\}$.
\item $\ol{M_{j'}^+}:=M_{j'}^{+} \bigcup \{  \lim_{t\to \infty} t\cdot c\mbox{; } c\in M_{j'}\}$.
\end{itemize}

Indeed, for a weighted stable map $c=[(C, \{s_0, s_\infty, s_1\}, e:\cO^{n+1}_C \to \cL)]$, we have $t\cdot c=[(C, \{s_0,  s_\infty, t\cdot s_1\}, e:\cO^{n+1}_C \to \cL)]$. If $c$ is in the blow-up locus, then at $ \lim_{t\to 0} t\cdot c$, the source contains a component $C_{01}$ of degree $0$ with $s_0, s_1\in C_{01}$  (note that $s_0\not=s_1$ due to the stability conditions). As well, $m_0=i_{j'}$ and the stability conditions imply $s_\infty\in  C_{01}$ and $m_\infty=d-i_{j'}$. On the other hand,  $ \lim_{t\to \infty} t\cdot c$ satisfies $s_1=s_\infty$ and so is a point in the sink of $M$.

Let $\cI$ denote the ideal sheaf of $M_{j'}$.  The $\CC^*$--action on $M$ induces a natural splitting $\cI=\cI^+\oplus\cI^-$. Here $\cI^+$ is the ideal of $\overline{M_{j'}^{+}}$ in $M$, and it admits a natural increasing filtration $\{\cI_l\}_{l> 0}$.  Each term $\cI_l$ can be described explicitly, in an \'{e}tale neighborhood  of ${M_{j'}}$  by $\CC^*$--invariant quasi-affine schemes,
as generated by the sections on which $\CC^*$ acts with weights $\geq l$.

 Consider the equivariant embedding $j: X_{(j, j'+1)} \hookrightarrow M$, so that, with the identifications from part (a), $j^{-1}\cI^+$ is the ideal of $\overline{X_{j'}^+}$ in $X_{(j, j'+1)}$. The ideal $j^{-1}\cI^+$ comes with a natural increasing filtration $\{j^{-1}\cI_l\}_{l> 0}$ compatible with the $\CC^*$--action.
 The construction of a smooth Deligne-Mumford stack structure  $\widetilde{X_{(j, j'+1)}}$ for the weighted blow-up given by this filtration  depends on the two technical conditions from \cite{noi5}, Lemma 2.5, which for $\{j^{-1}\cI_l\}_{l> 0}$ reduce to checking
 \bea j^{-1}(\cI_l\bigcap \cI^2)=  j^{-1}\cI_l\bigcap j^{-1}\cI^2.\eea
We may localize outside the sink and reduce the condition above to $j^{-1}(\cI_l\setminus \cI_{l+1})=j^{-1}(\cI_l)\setminus j^{-1}(\cI_{l+1})$ for each index $l\geq 1$, which indeed holds as $\cI^+_{|\overline{X_{j'}^+}}$ is graded by the weights of the $\CC^*$-action, and this grading is preserved when restricting to $X_{(j, j'+1)}$.

Finally, we can identify $\widetilde{X_{(j, j'+1)}}$ with $X_{(j, j'+\frac{1}{2})}$ as they are both embedded in $ \widetilde{M}$, they are both mapped birationally onto $X_{(j, j'+1)}$ and the maps restrict over the exceptional locus to $  E:= Y^-_{(j'-1, j')} \times_{X_{j'}} \overline{X_{j'}^+} \to  \overline{X_{j'}^+}.$

c)  Recall the weighted blow-up diagram
\bean  \diagram \Spec (\oplus_{n\geq 0} \cF^n) \rrto  &&\Spec (\oplus_{n\geq 0} \pi_* \cF^n)\\
E= Y^-_{(j'-1, j')} \times_{X_{j'}} \overline{X_{j'}^+}  \uto \rrto && \overline{X_{j'}^+} \uto \enddiagram \eean
where the second line represents the restriction over the exceptional locus.
  (\cite{noi5}, Lemma 2.8).

Let $F:=Y^-_{(j'-1, j')} \times_{X_{j'}} {X_{j'}^+}$, and let $U$ be the space obtained by removing the sink from $X_{(j,j')}$. From the description of the weighted blow-up $\pi$ in the proof of part (b) as induced by the weights of the $\CC^*$--action, we can identify the restriction to $U$  of the above diagram with the weighted blow-up diagram
\bea  \diagram \mbox{ Bl}^{w}_{X_{j'}}X_{j'}^- \times_{X_{j'}} {X_{j'}^+} \rrto  && X_{j'}^- \times_{X_{j'}} {X_{j'}^+}\\
F=Y^-_{(j'-1, j')} \times_{X_{j'}} {X_{j'}^+} \uto \rrto && {X_{j'}^+} \uto \enddiagram \eea
(conform \cite{noi5}, Lemma 2.8).

Here $\mbox{ Bl}^{w}_{X_{j'}}X_{j'}^- \times_{X_{j'}} {X_{j'}^+} = \Spec (\oplus_{n\geq 0} g^+_{j' *}\cO(n)), $ where $\cO(1)$ is canonically defined on $[( X^-_{j'}-X_{j'}) /\CC^*]$, the weighted projective fibration over $X_{j'}$, and $g^+_{j'}$ is the restriction of $\ol{g^+_{j' }}$ to $F$.

As a consequence, the normal bundle $\cN_{E|X_{(j,j'+\frac{1}{2})}}$ can be written as
\bean \cN_{E|X_{(j,j'+\frac{1}{2})}}\cong  \ol{g^+_{j' }}^*\cO(-1) \otimes \cO(-mD), \eean
where $D=E-F$ and as such $\cO(-D)=f_{j'}^{- *}\cO_{\ol{X^+_{j'}}}(-1)$.  Moreover, restriction over the fibres of $\ol{g^+_{j' }}$ gives $m=1$.

After dualizing and push-forward of the above equation by $f_{j'}^{- }$ we obtain:
\bea   f_{j'  *}^{-}\cN_{E|X_{(j,j'+\frac{1}{2})}}^\vee \cong  f_{j'  *}^{-}(\ol{g^+_{j' }}^*\cO(1) \otimes f_{j'}^{- *}\cO_{\ol{X^+_{j'}}}(1)) &\cong&   f_{j'  *}^{-}\ol{g^+_{j' }}^*\cO(1) \otimes \cO_{\ol{X^+_{j'}}}(1)  \\ &\cong&  \ol{q^+_{j' }}^*p_{j'  *}^{-}\cO(1) \otimes \cO_{\ol{X^+_{j'}}}(1),\eea
 while $\Spec (\oplus_{n\geq 0} p_{j'  *}^{-}\cO(n)))\cong X^{-}_{j'}$. This leads to the property of the affine fibration $A$ stated in part (c) of the theorem.  Furthermore, the normal bundle of the zero section in $X_{j'}^-$ is $ \cN_{X_{j'}|X}^-$ (consistent with \cite{noi5}, Lemma 2.10. for the given filtration), hence the formula for the normal bundle of the zero section in $A$.

Similarly,
\bea   \ol{g^+_{j' }}_*\cN_{E|X_{(j,j'+\frac{1}{2})}}^\vee \cong  \ol{g^+_{j' }}_*(\ol{g^+_{j' }}^*\cO(1) \otimes f_{j' }^{- *}\cO_{\ol{X^+_{j'}}}(1)) &\cong& \cO(1) \otimes \ol{g^+_{j' }}_*f_{j'  }^{- *}\cO_{\ol{X^+_{j'}}}(1) \\
&\cong& \cO(1) \otimes p_{j' }^{- *}\ol{q^+_{j' }}_*\cO_{\ol{X^+_{j'}}}(1).\eea
The weighted projective fibration $\ol{X^+_{j'}}=P^w(X^+_{j'}\oplus \cO_{X_j'})$ corresponds to the affine fibration $\Spec (\oplus_{n\geq 0} \ol{q^+_{j' }}_* \cO_{X^{-}_{j'}}(n)) \cong X^+_{j'}\oplus \cO_{X_j'}$. The normal bundle of the zero section in this affine bundle is $\cN_{X_{j'}|X}^+\oplus \cO_{X_{j'}}$. Hence the properties of the fibration $B$ stated in part (c) of the theorem.

\end{proof}

\begin{corollary} \label{basisofXpyramid} With the notations introduced earlier,  we have \bea \cU_{(j'-1,j')}(X) \cong X_{(j'-1,j')} \mbox{  and } \eea
\bea \cU_{(j'-1,j'+1)}(X) \cong X_{(j'-1, j'+\frac{1}{2})}  \times_{X_{(j'-1,j'+1)}}   X_{(j'-\frac{1}{2}, j'+1)} \cong X_{(j'-\frac{1}{2}, j'+\frac{1}{2})}.    \eea
\end{corollary}

\begin{proof}
At the level of moduli problems of weighted stable maps, the morphism \bea  \varphi_{(j,j')}:  \cU_{(j,j')}(X) \to X_{(j,j')}\eea corresponds to contracting those components of the domain curves which do not contain the section $s_1$. However, in the case of both spaces  $\cU_{(j'-1,j')}(X)$ and $X_{(j'-1,j')}$ the curves parametrized are $\PP^1$-s, hence $\cU_{(j'-1,j')}(X) \cong X_{(j'-1,j')}$.

On the other hand the morphism \bea \varphi_{(j'-1,j'+1)}:  \cU_{(j'-1,j'+1)}(X) \to X_{(j'-1,j'+1)} \eea
is a composition of two weighted blow-ups with blow-up loci $\overline{X^+_{j'}}$ and  $\overline{X^-_{j'}}$. Blowing up $X_{(j'-1,j'+1)}$ along $\overline{X^+_{j'}}$ yields $X_{(j'-\frac{1}{2},j'+1)}$ while blowing up $X_{(j'-1,j'+1)}$ along $\overline{X^-_{j'}}$ yields $X_{(j'-1,j'+\frac{1}{2})}$. The two blow-up loci intersect transversely along $X_{j'}$. Hence
\bea \cU_{(j'-1,j'+1)}(X) \cong X_{(j'-1, j'+\frac{1}{2})}  \times_{X_{(j'-1,j'+1)}}   X_{(j'-\frac{1}{2}, j'+1)} \cong X_{(j'-\frac{1}{2}, j'+\frac{1}{2})}.    \eea

\end{proof}

We obtained a family of birational morphisms

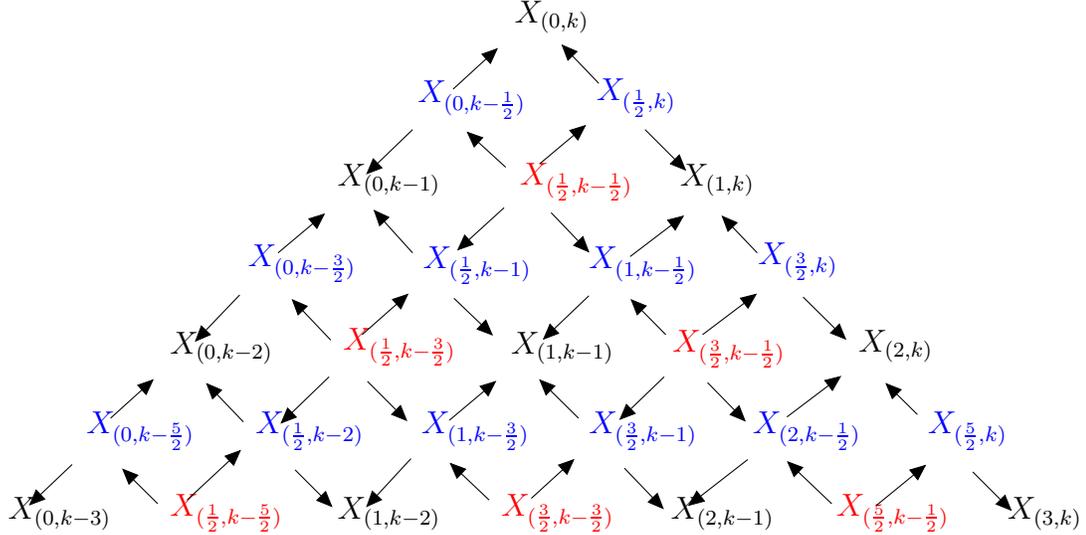
\definecolor{ffqqqq}{rgb}{1,0,0}
\definecolor{qqqqff}{rgb}{0,0,1}
\begin{tikzpicture}[line cap=round,line join=round,>=triangle 45,x=0.55cm,y=0.55cm]
\clip(-2.62,-7.88) rectangle (24.24,5.53); \label{grand_pyramid}
\draw [color=qqqqff](7.82,3.82) node[anchor=north west] {$X_{(0,k-\frac{1}{2})}$};
\draw [color=qqqqff](12.14,3.87) node[anchor=north west] {$X_{(\frac{1}{2}, k)}$};
\draw (5.88,1.8) node[anchor=north west] {$X_{(0, k-1)}$};
\draw [color=ffqqqq](10.29,1.84) node[anchor=north west] {$X_{(\frac{1}{2}, k-\frac{1}{2})}$};
\draw (14.16,1.8) node[anchor=north west] {$X_{(1,k)}$};
\draw [color=qqqqff](3.72,-0.14) node[anchor=north west] {$X_{(0, k-\frac{3}{2})}$};
\draw [color=qqqqff](7.95,-0.18) node[anchor=north west] {$X_{(\frac{1}{2}, k-1)}$};
\draw (10.16,5.76) node[anchor=north west] {$X_{(0,k)}$};
\draw [color=qqqqff](11.96,-0.18) node[anchor=north west] {$X_{(1,k-\frac{1}{2})}$};
\draw [color=qqqqff](16.05,-0.09) node[anchor=north west] {$X_{(\frac{3}{2}, k)}$};
\draw (1.83,-2.25) node[anchor=north west] {$X_{(0,k-2)}$};
\draw [color=ffqqqq](6.02,-2.16) node[anchor=north west] {$X_{(\frac{1}{2},k-\frac{3}{2})}$};
\draw (10.07,-2.25) node[anchor=north west] {$X_{(1,k-1)}$};
\draw [color=ffqqqq](13.98,-2.21) node[anchor=north west] {$X_{(\frac{3}{2},k-\frac{1}{2})}$};
\draw (18.48,-2.21) node[anchor=north west] {$X_{(2,k)}$};
\draw [color=qqqqff](-0.19,-4.19) node[anchor=north west] {$X_{(0, k-\frac{5}{2})}$};
\draw [color=qqqqff](3.9,-4.19) node[anchor=north west] {$X_{(\frac{1}{2}, k-2)}$};
\draw [color=qqqqff](7.91,-4.19) node[anchor=north west] {$X_{(1, k-\frac{3}{2})}$};
\draw [color=qqqqff](11.96,-4.19) node[anchor=north west] {$X_{(\frac{3}{2}, k-1)}$};
\draw [color=qqqqff](15.92,-4.19) node[anchor=north west] {$X_{(2, k-\frac{1}{2})}$};
\draw [color=qqqqff](20.15,-4.19) node[anchor=north west] {$X_{(\frac{5}{2},k)}$};
\draw (-2.08,-6.26) node[anchor=north west] {$X_{(0,k-3)}$};
\draw [color=ffqqqq](1.83,-6.17) node[anchor=north west] {$X_{(\frac{1}{2}, k-\frac{5}{2})}$};
\draw (5.88,-6.26) node[anchor=north west] {$X_{(1,k-2)}$};
\draw [color=ffqqqq](9.84,-6.17) node[anchor=north west] {$X_{(\frac{3}{2}, k-\frac{3}{2})}$};
\draw (13.94,-6.26) node[anchor=north west] {$X_{(2,k-1)}$};
\draw [color=ffqqqq](17.94,-6.17) node[anchor=north west] {$X_{(\frac{5}{2}, k-\frac{1}{2})}$};
\draw (22.08,-6.26) node[anchor=north west] {$X_{(3,k)}$};
\draw [->] (8.94,3.36) -- (10.02,4.35);
\draw [->] (12.45,3.5) -- (11.55,4.44);
\draw [->] (7.95,2.37) -- (6.83,1.29);
\draw [->] (4.71,-0.6) -- (5.84,0.35);
\draw [->] (3.77,-1.63) -- (2.69,-2.76);
\draw [->] (0.66,-4.6) -- (1.7,-3.66);
\draw [->] (-0.28,-5.73) -- (-1.32,-6.72);
\draw [->] (11.03,1.55) -- (12.15,2.49);
\draw [->] (13.58,2.37) -- (14.57,1.38);
\draw [->] (16.28,-0.6) -- (15.42,0.3);
\draw [->] (17.4,-1.68) -- (18.44,-2.71);
\draw [->] (20.15,-4.51) -- (19.38,-3.75);
\draw [->] (21.5,-5.91) -- (22.44,-6.85);
\draw [->] (10.2,1.47) -- (9.26,2.33);
\draw [->] (11.31,0.36) -- (12.23,-0.6);
\draw [->] (10.16,0.48) -- (9.03,-0.55);
\draw [->] (12.21,-1.65) -- (11.1,-2.71);
\draw [->] (14.03,-3.61) -- (12.95,-4.65);
\draw [->] (16.05,-5.61) -- (14.61,-6.73);
\draw [->] (13.22,-0.69) -- (14.52,0.3);
\draw [->] (14.97,-2.58) -- (16.28,-1.59);
\draw [->] (17.01,-4.57) -- (18.32,-3.58);
\draw [->] (19.16,-6.72) -- (20.37,-5.73);
\draw [->] (7.95,-0.6) -- (7.01,0.44);
\draw [->] (5.97,-2.71) -- (5.03,-1.72);
\draw [->] (3.86,-4.69) -- (2.96,-3.75);
\draw [->] (1.79,-6.67) -- (0.93,-5.82);
\draw [->] (8.96,-1.71) -- (9.9,-2.61);
\draw [->] (6.87,-3.72) -- (7.83,-4.68);
\draw [->] (5.12,-5.86) -- (6.02,-6.76);
\draw [->] (6.74,-2.58) -- (7.86,-1.59);
\draw [->] (8.85,-4.65) -- (9.98,-3.7);
\draw [->] (10.83,-6.54) -- (11.87,-5.59);
\draw [->] (14.07,-2.58) -- (13.22,-1.72);
\draw [->] (11.96,-4.6) -- (11.01,-3.66);
\draw [->] (9.8,-6.58) -- (8.85,-5.68);
\draw [->] (15.09,-3.75) -- (16.01,-4.69);
\draw [->] (13.08,-5.82) -- (14.03,-6.81);
\draw [->] (18.08,-6.67) -- (17,-5.64);
\draw [->] (7.91,-5.59) -- (6.83,-6.76);
\draw [->] (5.93,-3.61) -- (4.76,-4.74);
\draw [->] (2.51,-6.58) -- (3.81,-5.37);
\end{tikzpicture}

which we would like to compare to that of Corollary    \ref{1stpyramid}.

\begin{notation} \label{notation_morphisms}

For every $j, j', l, l'$ such that  $0\leq j \leq l < l'\leq j' \leq k$ we will employ the following notations for the natural morphisms between moduli spaces:
 \begin{tabular}{cc} $m^{j, j'}_{l, l'}: M_{(j, j')}(X) \to M_{(l, l')}(X),$  & $u^{j, j'}_{l, l'} : \cU_{(j, j')}(X) \to \cU_{(l, l')}(X),$ \\
$p_{(j, j')}: \cU_{(j, j')}(X) \to M_{(j, j')}(X), $ &
$\varphi_{(j, j')}: \cU_{(j, j')}(X) \to X_{(j, j')}$.
 \end{tabular}
For suitable choices of  $j, j', l, l'$, we will denote $v^{j, j'}_{l, l'} : X_{(j, j')}(X) \to X_{(l, l')}$.
For simplicity of notations, in cases when there is no danger of confusion we may omit the indices of the morphisms, preserving only the those of the domain and target spaces.
\end{notation}

\begin{theorem} For a pair of positive integers $(j, j')$ such that $0<j<j'<k$, consider the triangle formed by all the spaces $X_{(l,l')}$ above such that
\bea   0\leq j \leq l < l'\leq j' \leq k. \eea
We denote by   $I(j,j')$ the inverse family generated by all the morphisms between these spaces found in the triangle above, their compositions and fibred products.

There are natural isomorphisms as follows:
\begin{itemize}
\item[a)] $\cU_{(j, j'+1)}(X) \cong \cU_{(j, j')}(X)\times_{X_{(j'-1,j')}}X_{(j'-\frac{1}{2}, j'+\frac{1}{2})}$.
\item[b)] The universal family $\cU_{(j, j')}(X)$ over $M_{(j, j')}(X)$ is isomorphic with the inverse limit of the inverse family  $I(j,j')$.
\end{itemize}
\end{theorem}

\begin{proof}


We first claim that $\cU_{(j,j'+1)}(X)$ is naturally isomorphic to
 \bean \label{isom1}   m^{j, j'+1 *}_{j, j'} \cU_{(j,j')}(X) \times_{m^{j, j'+1 *}_{j'-1, j'}\cU_{(j'-1,j')}(X)} m^{j, j'+1 *}_{j'-1, j'+1} \cU_{(j'-1,j'+1)}(X). \eean
Indeed, within each fibre over a point $b$ in $M_{(j,j'+1)}(X)$, both maps \bea \cU_{(j,j'+1)}(X) &\to & m^{j, j'+1 *}_{j, j'} \cU_{(j,j')}(X) \mbox{ and }  \\ m^{j, j'+1 *}_{j'-1, j'+1} \cU_{(j'-1,j'+1)}(X) &\to & m^{j, j'+1 *}_{j'-1, j'} \cU_{(j'-1,j')}(X)\eea  contract exactly those irreducible components $\cC_\infty$ which contain $s_\infty(b)$ and such that
 $\cC_\infty \setminus s_\infty(b)$ is mapped into $X^+_{j'}$.
The maps \bea  \cU_{(j,j'+1)}(X) & \to &  m^{j, j'+1 *}_{j'-1, j'+1}\cU_{(j'-1, j'+1)}(X) \mbox{ and }  \\ m^{j, j'+1 *}_{j, j'}\cU_{(j,j')}(X) & \to & m^{j, j'+1 *}_{j'-1, j'} \cU_{(j'-1,j')}(X)\eea contract exactly those (unions of) components $\cC'$ which do not contain $s_\infty(b)$ and whose image in $X$ is sent by the moment map to $[0, i_{j'-1}]$.
Thus $\cU_{(j,j'+1)}(X)$ can be obtained by gluing the complements of the contraction loci of the spaces $m^{j, j'+1 *}_{j, j'} \cU_{(j,j')}(X)$ and $m^{j, j'+1 *}_{j'-1, j'+1} \cU_{(j'-1,j'+1)}(X)$ along their image in  $ {m^{j, j'+1 *}_{j'-1, j'} \cU_{(j'-1,j')}(X)}$.

Furthermore, by definition, the pullback $m^{j, j'+1 *}_{j'-1, j'+1} \cU_{(j'-1,j'+1)}(X) $ is naturally isomorphic to
\bean \label{isom2}
  m^{j, j'+1 *}_{j'-1, j'}\cU_{(j'-1,j')}(X) \times_{m^{j'-1, j'+1 *}_{j'-1, j'}\cU_{(j'-1,j')}(X)} \cU_{(j'-1,j'+1)}(X), \eean
and on the other hand
\bean \label{isom3}
 m^{j, j'+1 *}_{j, j'} \cU_{(j,j')}(X) \cong  \cU_{(j,j')}(X)  \times_{ \cU_{(j'-1,j')}(X) } m^{j'-1, j'+1 *}_{j'-1, j'}\cU_{(j'-1,j')}(X) \eean
follows directly from the isomorphism  \bea M_{(j,j'+1)}(X) \cong  M_{(j,j')}(X) \times_{M_{(j'-1,j')}(X)}  M_{(j'-1,j'+1)}(X) \eea  (see Corollary \ref{1stpyramid}).
Putting together  (\ref{isom1}), (\ref{isom2}), and (\ref{isom3}) we obtain a natural isomorphism
\bean \label{isom4} \cU_{(j,j'+1)}(X) \cong  \cU_{(j,j')}(X) \times_{\cU_{(j'-1,j')}(X)}  \cU_{(j'-1,j'+1)}(X). \eean

By Corollary \ref{basisofXpyramid}, we have $\cU_{(j'-1,j')}(X) \cong X_{(j'-1,j')}$ and
\bea \cU_{(j'-1,j'+1)}(X) \cong X_{(j'-1, j'+\frac{1}{2})} \times_{X_{(j'-1,j'+1)}}   X_{(j'-\frac{1}{2}, j'+1)} \cong X_{(j'-\frac{1}{2}, j'+\frac{1}{2})}.    \eea

Hence equation (\ref{isom4}) becomes: \bea \cU_{(j,j'+1)}(X) \cong  \cU_{(j,j')}(X) \times_{X_{(j'-1,j')}}   X_{(j'-\frac{1}{2}, j'+\frac{1}{2})},  \eea
which proves part (a).

 To prove that the universal family $\cU_{(j, j')}(X)$ over $M_{(j, j')}(X)$ is isomorphic with the inverse limit of the inverse family  $I(j,j')$, we first show that there exist  natural maps  $\cU_{(j, j')}(X) \to X_{(l, l')}$  to all the spaces in the family $I(j, j')$. Indeed, the weight choices for $\cU_{(j, j')}(X)$
\bea  a_0= 1-ai_j \mbox{, } a_\infty = 1-a(d-i_{j'})\mbox{, } 0\leq a_1 <1-da \eea
are larger than the respective weights for $X_{(l-\frac{1}{2}, l'+\frac{1}{2})}  \hookrightarrow  \overline{M}_{0, \mathcal{A}(i_{l}-\frac{1}{2},i_{l'}+\frac{1}{2}, i_{l'}-i_{l}+\frac{1}{4})}{(\mathbb P^n, d, a)}$  whenever $0\leq j < l \geq l'< j' \leq k$:
\bea  \begin{array}{lll}   a_0:= 1-(i_{l'}-\frac{1}{2})a\mbox{, } a_\infty = 1-(d-i_{l}-\frac{1}{2})a \mbox{, } a_1=(i_{l'}-i_{l}-\frac{3}{4})a,\end{array}\eea
 which confirms the existence of a morphism $\cU_{(j, j'+1)}(X) \to X_{(l+\frac{1}{2}, l'-\frac{1}{2})}$ for the given ranges of $l, l'$.  Furthermore, each space in the family $I(j, j')$ is the targer of a map in the inverse system, with domain of the form $X_{(l+\frac{1}{2}, l'-\frac{1}{2})}$ for suitable $l, l'$ as above.

To finalize part (b), we apply induction on $j'>j$: In the case when $j'=j+1$, we know
$\cU_{(j, j+1)}(X) \cong X_{(j, j+1)}$. We now assume that $\cU_{(j, j')}(X)$ does
 satisfy the universal property of the inverse limit of the inverse family  $I(j,j')$, and prove the same for  $\cU_{(j, j'+1)}(X)$
 and the inverse family  $I(j,j'+1)$. Indeed, any scheme having compatible morphisms into all the spaces $X_{(l,l')}$ of $I(j,j'+1)$ will, in particular, have a natural morphism to $\cU_{(j, j')}(X)$ (by the induction hypothesis), and to $X_{(j'-\frac{1}{2}, j'+\frac{1}{2})}$, and hence, by part (a), to $\cU_{(j, j'+1)}(X)$. The naturality insures that $\cU_{(j, j'+1)}(X)$ is indeed the inverse limit of the family.


\end{proof}

\section{Computations in equivariant $K$--theory}

Let $\beta$ be the class of the map $\PP^1 \to X$ given by $t\to t\cdot x$ for $x$ generic in $X$. In the previous section we have described in detail the fixed locus $ M_{(0,k)}(X)=M_{\omega, 2}$, where $\omega =((u_0, u_\infty), \beta)$, and $u_0, u_\infty$ are the source and sink of the graph associated to the action. We will now discuss  the contribution of the fixed locus $M_{\omega, 2} = M_{(0,k)}(X)$ to the computation of the virtual fundamental class of $\overline{M}_{0,2}(X,\beta)$. This contribution consists in two components: a fixed part $[M_{\omega, 2}]^{vir}$ and a moving part contribution $e^{\CC^*}(N_{M_{\omega, 2}}^{vir})$. The fixed part of the virtual fundamental class can be recovered as a special case of Theorem 6.3 in \cite{andrei}:

Consider the Cartesian diagram

\bea \diagram  M_{(0,k)}(X) \rto \dto & \prod_{j=0}^{k-2} M_{(j, j+2)}(X) \dto^{p} \\
    \Delta:\prod_{j=0}^{k-1} M_{(j, j+1)}(X) \rto & \prod_{j=0}^{k-2} M_{(j, j+1)}(X)\times M_{(j+1, j+2)}(X) \enddiagram  \eea

 Here $\Delta$ is the diagonal map and $p= (m^{j,j+2}_{j,j+1}, m^{j,j+2}_{j+1,j+2})_j $. In this context, \bea [M_{(0,k)}(X)]^{vir}= p^*([\Delta]).\eea This is indeed a cycle in $M_{(0,k)}(X)$ as $M_{(0,k)}(X)= p^{-1}(\mbox{ Im }\Delta)$.

 For the remainder of this section we will focus on the moving part.

 The virtual normal bundle  $N_{M_{(0,k)}(X)}^{vir}$ satisfies \bea [N_{M_{(0,k)}(X)}^{vir}]= \chi_{\neq 0}(R^{\bullet} p_{(0,k) *}( ev^*(\cT_X))), \eea
 in the equivariant $K$-theory of $M_{(0,k)}(X)$. The equivariant $K$-groups come with a grading given by the $\CC^*$-action, and for any complex $N^{\bullet}$, the term $\chi_{\neq 0}(N^{\bullet})$ represents the positive degree part of  $\chi(N^{\bullet})$.
  The formula above follows directly from \cite{ber}, \cite{graber_pand} together with \cite{andrei}. In general one should also consider the contribution from the relative cotangent complex of $m: M_{0,2}(X, \beta) \to \mathfrak{M}_{0,2}$, where $ \mathfrak{M}_{0,2}$ is the Artin stack of rational curves. However  the relative cotangent complex above does not contribute to the moving part for $M_{(0,k)}(X)$ because, given any map parametrized by a point in  $M_{(0,k)}(X)$, the deformations of the domain curve that keep $s_0$ and $s_\infty$ fixed are $\CC^*$-equivariant.

Our strategy for computing $\chi_{\neq 0}(R^{\bullet} p_{(0,k) *}( ev^*(\cT_X)))$ will be the following: we will write the class of $ev^*(\cT_X)$ in the equivariant K-theory of $ \cU_{(0,k)}(X)$ as a sum of classes supported on the fixed point locus of $ \cU_{(0,k)}(X)$ together with classes which are mainly pulled-back from $ M_{(0,k)}(X)$ or some substrata of $ M_{(0,k)}(X)$. These last classes will not contribute to  $\chi_{\neq 0}(R^{\bullet} p_{(0,k) *}( ev^*(\cT_X)))$   as their push-forward via $p_{(0,k)}$ will have trivial $\CC^*$-action.

 \begin{proposition}\label{induction}

 Recall that $\varphi_{(l,l')}: \cU_{(l,l')} \to X_{(l,l')}$ is the natural "evaluation map" and for $j\leq l<l'\leq j'$, we denoted by $u^{j,j'}_{l,l'}: \cU_{(j,j')}(X) \to \cU_{(l,l')}(X)$ the natural maps between the universal families over $m^{j,j'}_{l,l'}: M_{(j,j')}(X) \to M_{(l,l')}(X)$.

 Let $g_{ll'}^{jj'}:= \varphi_{(l,l')} \circ u_{j,j'}^{l,l'}$. For any $(j+1) <l \leq l' \leq j'$, the following equation holds in the equivariant $K$-groups of $ \cU_{(j,j')}(X)$: $$[g_{(j+1)l}^{jj'*}(\cT_{X_{(j+1,l)}})]-[g_{jl}^{jj'*}(\cT_{X_{(j,l)}})]=[g_{(j+1)l'}^{jj'*}(\cT_{X_{(j+1,l')}})]-[g_{jl'}^{jj'*}(\cT_{X_{(j,l')}})],$$ where $\cT_{X_{(j,l)}}$ is the tangent bundle to $X_{(j,l)}$.
 \end{proposition}

 \begin{proof}

 We will identify the component $X_{j+1}$ of the fixed locus in $X$ with its preimage in $X_{(j,l)}$. Using the stratification given by the $\CC^*$-action on $X_{j,l}$ we denoted $$\overline{X^-_{j+1}} = \ol{\{ x \in X_{(j,l)} \mbox{ such that } \lim_{t \to \infty} tx \in X_{j+1}\}}.$$ There is an open set of $X_{j,l}$ containing $\overline{X}_{j+1}^-$ on which the natural birational map between $X_{(j,l)} $ and $X_{(j,l')}$ is an isomorphism. On the other hand $g_{(j+1)l}^{jj'}*(\cT_{X_{(j+1,l)}})$ and $g_{jl}^{jj'}(\cT_{X_{(j,l)}})$ are isomorphic outside $g_{jl}^{j(j'-1)}(\overline{X^-_{j+1}})$, and also in a neighborhood $U$ around  $g_{jl}^{j(j'-1)}(\overline{X^-_{j+1}})$,  $$g_{(j+1)l}^{jj'*}(\cT_{X_{(j+1,l)}})_{|U}=g_{(j+1)l'}^{jj'*}(\cT_{X_{(j+1,l')}})_{|U}$$ and $$g_{jl}^{jj'*}(\cT_{X_{(j,l)}})_{|U}= g_{jl'}^{jj'*}(\cT_{X_{(j,l')}})_{|U}.$$

 \end{proof}

The following proposition will be our main tool in comparing classes of tangent bundles and their pull-backs:

 \begin{proposition}\label{Chern}

 Let $Z$ denote a smooth Deligne-Mumford stack and let $Y$ be a smooth substack. For a suitable increasing filtration of the ideal sheaf of $Y$ in $X$, we denote by
 $f: \widetilde{Z} \to Z$ the weighted blow-up of $Z$ along $Y$ (as defined in \cite{noi5}, section 2.1). Let $N_{Y|Z}= \oplus_w Q_w$ be the decomposition of the normal bundle after the weights associated to the weighted blow-up, where $w$ denotes the weight of $Q_w$.

 We will denote  by $g: E \to Y$ the restriction of $f$ to the exceptional divisor and by $i: E \hookrightarrow \widetilde{Z}$ the embedding of the exceptional divisor into $\widetilde{Z}$.

 The following equation holds in the $K$-theory of $\widetilde{Z}$ with rational coefficients:
  \bea [\cT_{\widetilde{Z}}]-[f^*(\cT_Z)]= i_*([\cO_E(E)] -\sum_w\sum_{0 \leq j < w}[g^*(Q_w)\otimes \cO_E(-jE)]). \eea  Moreover, if $\CC^*$ acts on $Z$ and the embedding of $Y$ in $Z$ is equivariant, then weighted blow-up map is equivariant then the same formula holds in the equivariant $K$-theory with rational coefficients.

 \end{proposition}

 \begin{proof}
   The proposition follows directly from Proposition 2.12 and Theorem 2.13 from \cite{noi5} after translating from the Chow classes to classes in $K$-theory,
  and using the short exact sequences  on $\widetilde{Z}$:
  \bea 0 \to \cO(-(j+1)E) \longrightarrow \cO(-jE) \longrightarrow \cO_E(-jE) \to 0 \eea
 for $0 \leq j < w.$
\end{proof}

 We apply the the above proposition in the context given by the following commutative diagram of weighted blow-ups

 \bea
\diagram & & X_{(j, j+2)} & & \\
& X_{(j, j+\frac{3}{2})} \urto^{v^{j, j+\frac{3}{2}}_{j, j+2}} \dlto_{v^{j, j+\frac{3}{2}}_{j, j+1}} & &
X_{(j+\frac{1}{2}, j+2)} \ulto_{v^{j+\frac{1}{2},j+2}_{j,j+2}}\drto^{v^{j+\frac{1}{2}, j+2}_{j+1, j+2}} & \\
X_{(j, j+1)} \dto^{\cong} & & X_{(j+\frac{1}{2}, j+\frac{3}{2})} \uuto^{\varphi_{(j,j+2)}} \dlto \dto^{\cong} \drto
\llto_{u^{j, j+2}_{j, j+1}} \ulto \urto \rrto^{u^{j, j+2}_{j+1, j+2}}  & & X_{(j+1, j+2)} \dto_{\cong} \\
\cU_{(j, j+1)}(X) \dto^{p_{(j,j+1)}}  & {m^{j,j+2 *}_{j,j+1}\cU_{(j, j+1)}(X)} \drto^{q} \lto_{H} & \cU_{(j, j+2)}(X) \dto^{p_{(j,j+2)}} \rto \lto^{ b} & {m^{j,j+2 *}_{j,j+1}\cU_{(j+1, j+2)}(X)} \dlto \rto  & \cU_{(j+1, j+2)}(X) \dto_{p_{(j+1,j+2)}} \\
M_{(j, j+1)}(X) & & M_{(j, j+2)}(X) \llto_{m^{j,j+2}_{j,j+1}} \rrto^{m^{j,j+2}_{j+1,j+2}} & & M_{(j+1, j+2)}(X)
\enddiagram
\eea

    With the identification $\cU_{(j,j+2)}\simeq X_{(j+\frac{1}{2},j+\frac{3}{2})}$, we can write $\varphi_{(j,j+2)}: \cU_{(j,j+2)} \to X_{(j,j+2)}$ as the composition of 2 weighted blow-ups. As well, using $\cU_{(j,j+1)}\simeq X_{(j,j+1)}$, note that $u_{j,j+1}^{j,j+2}$ factorizes through $m_{j,(j+1)}^{j,(j+2)*}(\cU_{(j,j+1)})$, where $ \cU_{(j,j+2)} \to m_{j,j+1 }^{j,j+2 *}(\cU_{(j,j+1)})$ is the standard blow-up of a codimension 2 stratum and $ m_{j,j+1}^{j,j+2 *}(\cU_{(j,j+1)}) \to \cU_{(j,j+1)}$ is a weighted blow-up.

  $X_{j+1}$ is a component of the fixed locus in $ X_{(j,j+2)}$, included in the two blow-up loci $\overline{X^+_{j+1}}$ and $\overline{X^-_{j+1}}$. Let $\cN_{X_{j+1}|X_{j+1}^+}= \oplus_w Q_{w,j+1}^+$ and $\cN_{X_{j+1}|X_{j+1}^-}= \oplus_w Q_{w,j+1}^-$ denote the positive and negative parts of the normal bundle of $X_{j+1}$ in $ X_{(j,j+2)}$ with their weight decompositions. The two exceptional divisors
 $T_{j+1}^+$ and $T_{j+1}^-$ in $\cU_{(j,j+2)}$ intersect after $\PP^w(\cN_{X_{j+1|X_{j+1}^+}}) \times_{X_j} \PP^w(\cN_{X_{j+1|X_{j+1}^-}})$. We will denote by $i^+_{j+1}$, $i^-_{j+1}$ and $i^{0}_{j+1}$ the embeddings of $T_{j+1}^+$,  $T_{j+1}^-$ and $ T_{j+1}^+ \cap T_{j+1}^-$ respectively in $\cU_{(j,j+2)}$, and by $\oplus_w \widetilde{Q}_{w,j+1}^-$ and  $\oplus_w  \widetilde{Q}_{w,j+1}^+$ the pullbacks of the normal bundles of $\overline{X^+_{j+1}}$ and $\overline{X^-_{j+1}}$ to $T_{j+1}^+$ and  $T_{j+1}^-$ respectively.

 For every $0 \leq j \leq k-2$ we will denote by $\widetilde{T}_j^+$, $\widetilde{T}_j^-$ and $\widetilde{T_j}$ the pseudo-divisors given by the pull-backs to $\cU_{(0,k)}(X)$ of $T_j^+$, $T_j^-$ and $T_j^++T_j^-$ via $u^{0,k}_{j,j+2}$.

 The complex $$( \cO_{\cU_{(0,k)}(X)}\rightarrow u^{0,k*}_{j,j+2}\cO_{\cU_{(j,j+2)}(X)}(T_j^+)) = : \cO_{\widetilde{T}_j^+} $$ is supported on $\widetilde{T}_j^+$. Similarly we define $\cO_{\widetilde{T}_j^-}$ and $\cO_{\widetilde{T}_j}$. Let $D_j := \widetilde{T}_j^+ \cap \widetilde{T}_j^-$ and $\cO_{D_j}=
(\cO_{\cU_{(0,k)}(X)}\oplus \cO_{\cU_{(0,k)}(X)} \rightarrow u_{j,j+2}^{0, k *}\cO_{\cU_{(j,j+2)}(X)}(T_j^+) \oplus u_{j,j+2 }^{0,k *}\cO_{\cU_{(j,j+2)}(X)}(T_j^-))$. We will denote by 
$Q_{w|D_j}^{\pm}:= \oplus u_{j,j+2 }^{0,k *}\varphi_{(j,j+2)}^*(Q_w^{\pm})\otimes \cO_{D_j}$.


 \begin{proposition} \label{Kclasses}

 With the notations from above, we have:

a) The following relation holds on $\cU_{(j,j+2)}(X)$: \bea [T_{ \cU_{(j,j+2)}(X)}] &-& [\varphi^*_{(j,j+2)}\cT_{X_{(j,j+2)}}] =  \\ &=& i^+_{j+1 * }([\cO(T_{j+1}^+)]-\sum_w\sum_{0 \leq l < w}[\widetilde{Q}_{w,j+1}^- \otimes \cO(-lT_{j+1}^+)])\\ &+&  i^-_{j+1 * }( [\cO(T_{j+1}^-)] - \sum_w\sum_{0 \leq l < w}[\widetilde{Q}_{w,j+1}^+ \otimes \cO(-lT_{j+1}^-)]). \eea

 \bea b)[\varphi^*_{(j,j+2)}\cT_{X_{(j,j+2)}}]&-&[u_{j,j+1}^{j,j+2 *}\cT_{X_{(j,j+1)}}] =[\cO( s_{\infty})]-[u_{j,j+1}^{j, j+2 *}\cO( s_{\infty})] - \\ &-& i^{0}_{j+1 *}(\sum_w\sum_{0 \leq l < w}\sum_{1\leq m \leq w-l}[Q_{w,j+1}^+ \otimes \cO(mT_{j+1}^+-lT_{j+1}^-)])\\ &-& i^{0}_{j+1 *}(\sum_w\sum_{0 \leq l < w}\sum_{1\leq m \leq w-l}[Q_{w,j+1}^- \otimes \cO(mT_{j+1}^--lT_{j+1}^+)]) +\\
 &+& i^-_{j+1 *} c_{-}+ i^+_{j+1 *} c_+\eea where $c_-$, $c_+$ are classes pulled-back from the boundary divisor $E_{(j,j+2)}$ in $M_{(j,j+2)}(X)$ to $T_{j+1}^-$ and $T_{j+1}^+$, respectively. Both $T_{j+1}^-$ and $T_{j+1}^+$ are $\PP^1$-bundles over $E_{(j,j+2)}$, and $E_{(j,j+2)} \cong T_{j+1}^- \bigcap T_{j+1}^+$.

 We denoted by $\cO(s_{\infty})$ the line bundle of the divisor given by the image of $s_\infty$, in the relevant universal family. By a slight abuse of notation, we denoted by $Q_{w,j+1}^{\pm}$ the pull-back of $Q_{w, j+1}^{\pm}$ to $T_j^+ \cap T_j^-$.

 c) Let $P^+_{w,l,j+1}:= - p_{(j,j+2)_*}i^0_{j+1 *}(\sum_{1\leq m \leq w-l}[Q_{w,j+1}^+ \otimes \cO(mT_{j+1}^+-lT_{j+1}^-)])$. The pullback of $P^+_{w,l,j+1}$ to the exceptional divisor $E_{(j,j+2)}$ of $M_{(j,j+2)}(X)$ is \bea [Q_{w,j+1}^+ \otimes\cO(lT_{j+1}^+)]-[Q_{w,j+1}^+ \otimes\cO(-lT_{j+1}^-)].\eea Assuming that there are vector bundles $K_{w,j+1}^+$, $L_{j+1}^{\pm}$ on  $M_{(j,j+2)}(X)$ whose pullbacks to $E_{(j,j+2)}$  are $Q_{w,j+1}^+$ and $\cO(T_{j+1}^{\pm})$ respectively,
 and such that $\cO(E_{(j,j+2)})=\cO(L_{j+1}^{+}+L_{j+1}^{-})$
  then \bea P^+_{w,l,j+1}=[K_{w,j+1}^+ \otimes\cO(lL_{j+1}^+)]-[K_{w,j+1}^+ \otimes\cO(-lL_{j+1}^-)]. \eea

d) The following relation holds in the equivariant $K$-theory of $M_{(0,k)}(X)$:
\bea \chi_{\neq 0}(p_{(0,k)*}[ev^*(\cT_X)]) &= &\sum_w\sum_{0 \leq l<w}[Q_{w,0}^+ \otimes\cO(l s_0)]+ \sum_w\sum_{0 \leq l<w}[Q_{w,k}]^- \otimes\cO(l  s_{\infty})] \\& +&  \sum_{j=1}^{k-1}\sum_{w}\sum_{0 \leq l < w }P^+_{w,l,j} +\sum_{j=1}^{k-1}\sum_{w}\sum_{0 \leq l < w } P^-_{w,l,j}.\eea

\end{proposition}

 \begin{proof}

 a)This is just a direct application of Proposition \ref{Chern}. We use the fact that the strata $\overline{X_{j+1}^+}$ and $\overline{X_{j+1}^-}$ are transverse to each other and so that the normal bundle of the strict transform of $\overline{X_{j+1}^+}$ in $X_{(j+\frac{1}{2},j+2)}=Bl^w_{\overline{X_{j+1}^-}} X_{(j,j+2)}$ is the pull-back of the normal bundle of $\overline{X}_{j+1}^+$ in $X_{(j,j+2)}$.

 b) Denote by $\widetilde{\overline{X_{j+1}^-}}$ the strict transform of $\overline{X_{j+1}^-}$ in $X_{(j,j+\frac{3}{2})}$, and by $E$ the exceptional divisor for the morphism $ X_{(j,j+\frac{3}{2})} \to X_{(j,j+2)}$. The blow-down morphism $v^{j,j+\frac{3}{2}}_{ j,j+1}: X_{(j,j+\frac{3}{2})} \to X_{(j,j+1)}$ has the same exceptional divisor $E$. The normal bundle $N$ of $v^{j,j+\frac{3}{2}}_{ j,j+1}(\widetilde{\overline{X_{j+1}^-}})\cong \overline{X_{j+1}^-}$ in $X_{(j,j+1)}$ is the pull-back of the normal bundle of  $Y^-_{(j, j+1)}=p_{j,j+1}(v^{j,j+\frac{3}{2}}_{ j,j+1}(\widetilde{\overline{X_{j+1}^-}}))$ in $M_{(j,j+1)}(X)$.
 As the blow-up locus of $v_{ j,j+1}^{j,j+\frac{3}{2}*}$ is embedded in $v^{j,j+\frac{3}{2}}_{ j,j+1}(\widetilde{\overline{X_{j+1}^-}})$, by local computations
  \bea u^{j, j+2 *}_{j,j+1}N& =& v_{ j,j+\frac{3}{2}}^{j+\frac{1}{2},j+\frac{3}{2}*}v_{ j,j+1}^{j,j+\frac{3}{2}*}(N)=\oplus_w \widetilde{Q}_{w,j+1}^+\otimes\cO(wT_{j+1}^+)\\ &=& q^{- *}_{j,j+2}m^{j, j+2 *}_{j,j+1}\cN_{Y^-_{(j, j+1)}|M_{(j,j+2)}(X)},\eea where  $q^{- *}_{j,j+2}: T_{j+1}^- \to E_{(j,j+2)}$ is the projection to the exceptional divisor in $M_{j,j+2}(X)$.

 Furthermore, $T_{j+1}^+=v_{j,j+\frac{3}{2}}^{j+\frac{1}{2},j+\frac{3}{2} *}(E)$
 and so $Q_{w,j+1}^+\otimes \cO(-lT_{i+1}^-+(w-l)T_{j+1}^+)$ is the pull-back of a class from the exceptional divisor in $M_{j,j+2}(X)$.

  We use now that \bea i^-_{j+1 *}([\widetilde{Q}_{w,j+1}^+\otimes \cO(-lT_{j+1}^-)] &-& [\widetilde{Q}_{w,j+1}^+\otimes \cO(-lT_{i+1}^-+(w-l)T_{j+1}^+)])= \\&-&\sum _{0<m \leq w-l} i^{0}_{j+1 *}([Q_{w,j+1}^+\otimes \cO(-lT_{i+1}^-+mT_{j+1}^+)])\eea
  which follows from successive application of short exact sequences
  \bea 0 \to \cO((m-1)T^+_{j+1}) \longrightarrow \cO(mT^+_{j+1}) \longrightarrow \cO_{T^+_{j+1}}(mT^+_{j+1}) \to 0 \eea
  tensored with the bundles above.


  The formula now follows from the application of Proposition \ref{Chern} for the morphism $\cU_{(j,j+2)} \to X_{(j,j+\frac{3}{2})}$, combined with formula (a) and the computations above.


  c) The composition $p_{(j,j+2)}\circ i^0_{j+1} : T_{j+1}^+ \cap T_{j+1}^- \to M_{(j,j+2)}(X)$ is an isomorphism on its image $E_{(j,j+2)}$.

  Assume that there are vector bundles $K_{w,j+1}^+$, $L_{j+1}^{\pm}$ on  $M_{(j,j+2)}(X)$ whose pullbacks to $E_{(j,j+2)}$  are $Q_{w,j+1}^+$ and $\cO(T_{j+1}^{\pm})$ respectively, and such that $\cO(E_{(j,j+2)})=\cO(L_{j+1}^{+}+L_{j+1}^{-})$,
From the short exact sequence
\bea 0 \to \cO \longrightarrow \cO(E_{(j,j+2)}) \longrightarrow \cO_{E_{(j,j+2)}}(E_{(j,j+2)}) \to 0 \eea
tensored with relevant vector bundles, we get
 \bea [K^+_{w,j+1} \otimes \cO(kL_{j+1}^+)] &-&[K^+_{w,j+1} \otimes \cO(-kL_{j+1}^-)]= \\ &=& i_{*}(\sum_{k=1}^{w}\sum_{l=1}^{k-1}[Q^+_{w,j+1} \otimes \cO(kT_{j+1}^+-l\cO(E_{(j,j+2)}))], \eea
for the embedding $i: E_{(j,j+2)} \hookrightarrow M_{(j,j+2)}(X)$.

After a change of variables $m=k-l$, the last term above becomes $P^+_{w,l,j+1}$.

After replacing $M_{(j,j+2)}(X)$ by the normal bundle of $E_{(j,j+2)}$ in $M_{(j,j+2)}(X)$, the assumptions are satisfied in this new ambient space. We can now get the formula for the pull-back of $P^+_{w,l,j+1}$ to $E_{(j,j+2)}$ by applying the arguments above in this new context, and then pulling-back to $E_{(j,j+2)}$.

\bea 0 \to \cO(-(j+1)E) \longrightarrow \cO(-jE) \longrightarrow \cO_E(-jE) \to 0 \eea

 For part d) we apply induction on $k$.  At the initial step we use the relative Euler sequence for $\PP^1$-bundles to get the relative tangent bundle of $\cU_{(0,1)}(X) \to M_{(0,1)}(X)$ as $\cO(s_0) \oplus \cO(s_{\infty})/\cO_{\cU_{(0,1)}(X)}$.
 Then we apply formula (b) to get the formula for $k=2$. The induction step from $k$ to $k+1$ uses the same formula (b) for $j=k-1$ and Proposition \ref{induction}.

 \end{proof}

In order to obtain a nice closed formula for the moving part $e^{\CC^*}(\cN_{M_{0,k}(X)}^{vir})$,  we will assume that the conditions of part c) of the previous Proposition are satisfied. That is, we assume there are classes of vector bundles $K_{w,j}^{\pm}$ and line bundles $L_{w,j}^{\pm}$ on $M_{0,k}(X)$ such that $K_{w,j}^{\pm} \otimes \cO(D_j)= m^{0,k*}_{i,i+2}(Q_{w,j}^{\pm})\otimes \cO(D_j)$ and $L_{j}^{\pm} \otimes \cO(D_j)= m^{0,k*}_{i,i+2}\cO(T_{j}^{\pm})\otimes \cO(D_j)$. We denote by $c_K(t)$ the Chern polynomial of the vector bundle $K$. Let $p_N(t):=t^{\deg c_N}c_N(\frac{1}{t})$ and $\xi_j^\pm:=c_1(L_j^\pm)$.

\begin{theorem}

With the above notations $$e^{\CC^*}(\chi_{\neq 0}(R^{\bullet} p_{(0,k) *}( ev^*(\cT_X)))) = \prod_w \prod_{l=0}^{w-1}p_{Q_{w,0}^+}(w-lt-l\psi_1) \prod_w \prod_{l=0}^{w-1}p_{Q_{w,k}^-}((-w+l)t-l\psi_2) $$ $$\prod_{j=1}^{k-1} \frac{\prod_w \prod_{l=0}^{w-1}p_{K^+_{w,j}}((w-l)t +\xi_j^+)\prod_w \prod_{l=0}^{w-1}p_{K^-_{w,j}}((-w+l)t +\xi_j^-)}{\prod_w \prod_{l=0}^{w-1}p_{K^+_{w,j}}((w-l)t -\xi_j^-)\prod_w \prod_{l=0}^{w-1}p_{K^-_{w,j}}((-w+l)t -\xi_j^+)}.$$

\end{theorem}

\subsection{Moduli spaces $M_{\omega, 2}$ and their virtual classes.}

Let's consider first the case $\omega =((u_0, u_\infty), n\beta)$, where $\beta$ is the class of the generic curve. This case can be treated in the same way as  $\omega =((u_0, u_\infty), \beta)$ by composing the action of $\CC^*$ on $X$ with the map $ \CC^* \times X \to \CC^* \times X $ which is the quotient by the cyclic subgroup $\ZZ_n$  on the first component, and identity on the second. We now have a new action of $\CC^*$ on $X$. In this case there is a morphism $M_{(0,k)}(X) \to M_{\omega, 2}$ whose fiber over points in $ M_{\omega, 2}$ parameterizing irreducible curves is $ B\ZZ_n=[\mbox{point}/\ZZ_n]$, while over points parameterizing a chain of $k$ curves, the fiber is $B(\ZZ_n)^k$. For this reason we will denote $M_{\omega, 2}=:(M_{(0,k)}(X)/ \ZZ_n)$.

In this case, the fixed part of the virtual class works is calculated be the same formula as in the previous section, with  \bea \Delta: (M_{(j,j+1)}(X)/\ZZ_n) \to \prod_{j=0}^{k-2} (M_{(j,j+1)}(X)/\ZZ_n) \times (M_{(j+1,j+2)}/\ZZ_n) \eea and $p: \prod_{j=0}^{k-2}(M_{(j,j+2)}/\ZZ_n) \to  \prod_{j=0}^{k-2} (M_{(j,j+1)}(X)/\ZZ_n) \times (M_{(j+1,j+2)}/\ZZ_n)$

Similarly, the virtual normal bundle of $M_{\omega, 2}$ is
\bea \chi_{\neq 0}(R^{\bullet} p^n_{(0,k) *} ev^*(\cT_X)) &= &\sum_w\sum_{0 \leq l<nw}[Q_{w,0}^+ \otimes\cO(l s_0)]+ \sum_w\sum_{0 \leq l<nw}[Q^-_{w,l}] \otimes\cO(l  s_{\infty})] \\& +&  \sum_{j=1}^{k-1}\sum_{w}\sum_{0 \leq l < nw }P^+_{w,l,j} +\sum_{j=1}^{k-1}\sum_{w}\sum_{0 \leq l < nw } P^-_{w,l,j}.\eea

We consider now $M_{\omega, 2}$ in the case when $\omega =((u, v), n\beta)$ when $\beta=v-u$ is not the class of the generic orbit and $n \in \ZZ_{>0}$. The choice of the moment map $\mu:X \to \RR$ depends on the choice of the line bundle which induces a linear map $p: H^{1,1}(X)^\vee \to \RR$ such that $\mu= p \circ \mu^m$. If $p(u)=i_j$ and $p(v)=i_q$ are the $j$-th and the $q$-th walls of $\mu$, then there is a natural closed embedding $M_{\omega, 2} \subset (M_{(j,q)}(X)/\ZZ_n)$. Moreover, the fixed part of the virtual class of $M_{\omega, 2}$ is $h_{j+1,q}^*\PP^w(N_{X_{j}|X_{j}^+}) \cap h_{j,q-1}^*(\PP^w(N_{X_{q}|X_{q}^-})$ for the natural maps  $h_{j+1,q}: (M_{(j,q)}(X)/\ZZ_n) \to (M_{(j+1,q)}(X)/\ZZ_n) $ and $h_{j,q-1}: (M_{(j,q)}(X)/\ZZ_n) \to (M_{(j,q-1)} (X)/\ZZ_n)$.

The class of the virtual normal bundle is \bea && \chi_{\neq 0}(R^{\bullet} p^n_{(j,q) *} ev^*(\cT_X))= \\ &=&\sum_w\sum_{0 \leq l<nw}[Q_{w,j}^+ \otimes\cO(l \widetilde{T_j}^+)]-\sum_w\sum_{0 \leq l<nw}[Q_{w,j}^- \otimes\cO(-l \widetilde{T_j}^+)] \\&+& \sum_w\sum_{0 \leq l<nw}[Q_{w,q}]^- \otimes \cO(l \widetilde{T_q}^-)] -\sum_w\sum_{0 \leq l<nw}[Q_{w,q}]^+ \otimes \cO(-l \widetilde{T_q}^-)]\\ &+&  \sum_{s=j+1}^{q-1}\sum_{w}\sum_{0 \leq l < nw }P^+_{w,l,s} +\sum_{s=j+1}^{q-1}\sum_{w}\sum_{0 \leq l < w } P^-_{w,l,s}\eea

These formulae follow from those in Proposition \ref{Kclasses} together with the  product axiom of \cite{ber}.

Finally, we consider $M_{\omega, 2}$ in the case when $\omega =((u, v), \frac{n}{m}\beta)$ for $\beta=v-u$ and $n,m \in \ZZ_{>0}$. There is a smooth variety $X_m \subset X$ such that $X_m= \{x \in X | \ZZ_m \subset\mbox{Stab}(x)\}$. Moreover $\ZZ_m$ acts nontrivially on the fibers of the normal bundle $N_{X_m|X}$. As such the only contribution of $R^{\bullet} p_*(ev^*(TX))$ to the fixed part of the virtual class of $M_{\omega, 2}$ comes from $R^{\bullet} p_*(ev^*(TX_m))$. The formula for the class of the virtual normal bundle of $M_{\omega, 2}$ is \bea && \chi_{\neq 0}(R^{\bullet} p^{n/m}_{(j,q) *} ev^*(\cT_X)) = \\ &= &\sum_w\sum_{0 \leq l<\frac{n}{m}w}[Q_{w,j}^+ \otimes\cO(lm \widetilde{T_j}^+)]-\sum_w\sum_{0 \leq l<\frac{n}{m}w}[Q_{w,j}^- \otimes\cO(-lm \widetilde{T_j}^+)] \\&+& \sum_w\sum_{0 \leq l<\frac{n}{m}w}[Q_{w,q}]^- \otimes \cO(lm \widetilde{T_q}^-)] -\sum_w\sum_{0 \leq l<\frac{n}{m}w}[Q_{w,q}]^+ \otimes \cO(-lm \widetilde{T_q}^-)]\\& +&  \sum_{s=j+1}^{q-1}\sum_{w}\sum_{0 \leq l < \frac{n}{m}w}P^+_{w,lm,s} +\sum_{s=j+1}^{q-1}\sum_{w}\sum_{0 \leq l < \frac{n}{m}w} P^-_{w,lm,s}.\eea


\providecommand{\bysame}{\leavevmode\hbox
to3em{\hrulefill}\thinspace}

\end{document}